\documentclass{amsart}
\usepackage[utf8]{inputenc}
\usepackage[textsize=scriptsize]{todonotes}
\usepackage[expansion=false]{microtype}
\usepackage{amsmath, amsfonts, amssymb, amsthm, amsopn, amscd}
\usepackage{eucal}
\usepackage[pdfusetitle,unicode,hidelinks]{hyperref}
\usepackage{tikz-cd}
\usepackage{enumerate}
\usepackage{bbm}

\numberwithin{equation}{section}

\usepackage{etoolbox}
\newtoggle{draft}

\iftoggle{draft} {
\usepackage[margin=1.5in]{geometry}

\newcommand{\NB}[1]{\todo[color=gray!40]{#1}}

\newcommand{\tombubble}[1]{\todo[color=green!40]{#1}}
\newcommand{\tom}[1]{{\color{green!60!black}#1}}
\newcommand{\elden}[1]{{\color{blue!60!black}#1}}
\newcommand{\eldenbubble}[1]{\todo[color=blue!40]{#1}}
\newcommand{\paul}[1]{{\color{red!60!black}#1}}
\newcommand{\paulbubble}[1]{{\todo[color=red!40]{#1}}}

}{ 
\usepackage[margin=1in]{geometry}

\newcommand{\NB}[1]{}

\newcommand{\tombubble}[1]{}
\newcommand{\tom}[1]{}
\newcommand{\elden}[1]{}
\newcommand{\eldenbubble}[1]{}
\newcommand{\paul}[1]{}
\newcommand{\paulbubble}[1]{}
}
\newcommand{\compl}{{}^{{\kern -.5pt}\wedge}_{\ell}}
\newcommand{\comtwo}{{}^{{\kern -.5pt}\wedge}_{2}}
\newcommand{\comrho}{{}^{{\kern -.5pt}\wedge}_{\rho}}
\newcommand{\simp}{\mathrm{simp}}
\theoremstyle{plain}
\newtheorem*{theorem*}{Theorem}
\newtheorem{theorem}[equation]{Theorem}
\newtheorem{proposition}[equation]{Proposition}
\newtheorem{lemma}[equation]{Lemma}
\newtheorem{corollary}[equation]{Corollary}

\theoremstyle{definition}
\newtheorem{definition}[equation]{Definition}
\newtheorem{example}[equation]{Example}
\newtheorem{remark}[equation]{Remark}
\newtheorem{warning}[equation]{Warning}

\def\Map{\mathrm{Map}}
\def\map{\mathrm{map}}

\let\scr=\mathcal
\let\bb=\mathbb

\def\Z{\bb Z}
\def\R{\bb R}
\def\Q{\bb Q}

\def\A{\bb A}
\def\P{\bb P}

\def\N{\bb N}
\def\GL{\mathrm{GL}}
\def\1{\mathbbm{1}}
\def\h{\mathrm h}
\def\MGL{\mathrm{MGL}}

\def\eff{\mathrm{eff}}
\def\trdeg{\operatorname{trdeg}}

\def\ph{\mathord-}
\def\id{\mathrm{id}}
\def\op{\mathrm{op}}

\let\cat=\mathrm

\def\SH{\mathcal S\mathcal H}
\def\DM{\mathcal D\mathcal M}

\def\Mod{\mathcal M od}
\def\et{\mathrm{\acute et}}
\def\Cat{\mathcal{C}\mathrm{at}{}}

\def\Sm{{\cat{S}\mathrm{m}}}
\def\cd{\mathrm{cd}}
\def\Fun{\mathrm{Fun}}
\def\Span{\mathrm{Span}}
\def\vcd{\mathrm{vcd}}
\def\Fin{\cat F\mathrm{in}}
\DeclareMathOperator*{\colim}{colim}
\def\comp{\wedge}
\subjclass[2010]{Primary: 14F42, 55P42; Secondary: 14F20, 18N55, 18N60}
\def\Ind{\mathrm{Ind}}
\def\Pro{\mathrm{Pro}}

\newcommand{\wequi}{\simeq}
\def\adj{\leftrightarrows}

\newcommand{\Gm}{{\mathbb{G}_m}}
\newcommand{\Gmp}[1]{{\mathbb{G}_m^{\wedge #1}}}
\newcommand{\im}{\mathrm{im}}
\newcommand{\Mp}[1]{\1/#1}
\newcommand{\lra}[1]{\langle #1 \rangle}

\newcommand{\Spec}{\mathrm{Spec}}

\newcommand{\cof}{\mathrm{cof}}
\newcommand{\sk}{\mathrm{sk}}
\DeclareRobustCommand{\ul}{\underline}
\newcommand{\scomp}{\mathrm{sc}}

\def\Nis{\mathrm{Nis}}

\title{On \'etale motivic spectra and Voevodsky's convergence conjecture}

\author{Tom Bachmann}
\address{Department of Mathematics, Massachusetts Institute of Technology, Cambridge, MA, USA}
\email{tom.bachmann@zoho.com}

\author{Elden Elmanto}
\address{Department of Mathematics, Harvard University, Cambridge, MA, USA}
\email{elmanto@math.harvard.edu}

\author{Paul Arne {\O}stv{\ae}r}
\address{Department of Mathematics, University of Oslo, Norway}
\email{paularne@math.uio.no}

\begin{document}
	
\maketitle

\begin{abstract} We prove a new convergence result for the slice spectral sequence, 
following work by Levine and Voevodsky. This verifies a derived variant of Voevodsky's conjecture on convergence of the slice spectral sequence. This is, in turn,  a necessary ingredient for our main theorem: a Thomason-style \'etale descent result for the Bott-inverted motivic sphere spectrum, which generalizes and extends previous \'etale descent results for special examples of motivic cohomology theories.  
Combined with first author's \'etale rigidity results, we obtain a complete structural description of the \'etale motivic stable category.
\end{abstract}

\tableofcontents

\newpage
\section{Introduction}
The first goal of this paper is to push, as far as possible, the idea that inverting so-called ``Bott elements" on certain cohomology theories for schemes results in \'etale descent. The cohomology theories that we are interested in are those which are representable in Morel--Voevodsky's stable motivic homotopy category, and thus only satisfy Nisnevich descent. Results of this form are extremely useful when one tries to approximate the values of these cohomology theories by studying their \'etale sheafifications which are usually more tractable due to the presence of finer covers. For example, given appropriate finiteness hypotheses, there is a spectral sequence whose $E_2$-page involves \'etale cohomology, and converges to the \'etale sheafified version of these theories.  

The story for these motivic cohomology theories parallels the one for algebraic $K$-theory. Having proved the Nisnevich descent theorem for algebraic $K$-theory \cite{tt}, Thomason, in his ICM address \cite{thomason-icm}, explained how this ``(unleashed) a pack of new fundamental results for $K$-theory". One of them was the comparison result with \'etale $K$-theory \cite[\S 4]{thomason-icm}, referring to his seminal paper \cite{aktec}. More precisely, Thomason proved that Bott-inverted algebraic $K$-theory satisfy \'etale (hyper)descent. Whence, combined with the rigidity results of \cite{gabber-rigidity, suslin-closed}, there is a spectral sequence whose $E_2$-terms are \'etale cohomology groups, abutting to Bott-inverted algebraic $K$-theory. A recent treatment of this result using completely different methods was carried out in \cite{clausen-mathew}. 

The second goal of this paper, and a \emph{necessary ingredient} for the first, is to prove a variant of an important conjecture within the subject of stable motivic homotopy theory. In \cite{voe-open}, Voevodsky listed a collection of open problems in the subject, many of which are focused on understanding the slice filtration of a motivic spectrum. 

On the motivic spectrum representing algebraic $K$-theory, the slice filtration gives the most definitive construction of the elusive motivic spectral sequence whose $E_2$-page is motivic cohomology and abutting to algebraic $K$-theory. Such a spectral sequence was the subject of prior work of Bloch--Lichtenbaum, Levine and Friedlander--Suslin \cite{bloch-moving, levine-techniques, friedlander-suslin}. In contrast to previous constructions, this spectral sequence is constructed natively within the framework of the motivic stable homotopy category, analogous to the topological theory of Postnikov towers. It therefore extends to more general cohomology theories beyond algebraic $K$-theory and does not \emph{a priori} depend on difficult constructions involving moving algebraic cycles. We refer the reader to the survey article \cite[Section 2]{levine-appreciation} for an outline of Voevodsky's slice approach to the motivic spectral sequence.

As for any spectral sequence, an important issue is the problem of convergence. Voevodsky stated, as Conjecture 13 in \cite{voe-open}, that he expects convergence of his slice spectral sequence on the sphere spectrum, over any perfect field. Moreover, he expects that the slice spectral sequence induces a separated filtration on its homotopy sheaves; see also \cite[Conjecture 13]{voe-open}. This conjecture turns out to be false in general \cite[Remark on page 909]{levine2013convergence} and has since been modified by Levine to cover only fields of finite virtual cohomological dimension \cite[Conjecture 5]{levine2013convergence}. Roughly speaking, Levine conjectures that for such a field $k$, it is the \emph{$I(k)$-completed} slice tower which admits good convergence properties. Here, $I(k)$ is the so-called \emph{fundamental ideal}, defined to be the kernel of the rank map of the Grothendieck-Witt group of symmetric bilinear forms over $k$. Its action on any motivic spectrum is given by Morel's fundamental identification of $\pi_{0,0}$ of the motivic sphere spectrum with the Grothendieck-Witt group \cite{morel-pi0}. We will refer to  \cite[Conjecture 5]{levine2013convergence} as \emph{Levine-Voevodsky's slice convergence conjecture}.


\subsection{Slice convergence} In this paper, we offer a resolution of a version of Levine-Voevodsky's slice convergence conjecture. To state our results, recall that we have a map in the stable motivic homotopy category over a field $k$:
\[
\rho
\colon
\1
\rightarrow
\Gm,
\]
classifying the unit $-1\in k^{\times}$.

\begin{theorem}[Theorem~\ref{thm:convergence}]
\label{thm:conv-intro}
Let $k$ be a field of exponential characteristic $e$ and $\ell > 0$ coprime to $e$ such that $\vcd_\ell(k) < \infty$. 
Suppose that $E/\ell \in \SH(k)$ is bounded below\footnote{Throughout this paper we use homological notation for $t$-structures, as in \cite[\S1.2.1]{HA}.} in the homotopy $t$-structure and suppose that there exists $R \gg 0$ for which 
\[
E \xrightarrow{\rho^R} E \wedge \Gmp{R}
\] is zero; in other words it \emph{has bounded $\rho$-torsion}. Then, the filtration on $\pi_{i,j}(E)$ induced by Voevodsky's slice tower $f_{\bullet} E \rightarrow E$ is separated and exhaustive, 
i.e., the filtration is convergent.
\end{theorem}

We comment on how this is related to Levine-Voevodsky's slice convergence conjecture. There are two phenomenon which led Levine to formulate \cite[Conjecture 5]{levine2013convergence} in place of \cite[Conjecture 13]{voe-open}:
\begin{enumerate}
\item Levine showed in \cite{levine-gw} that the induced filtration on the zero-th homotopy sheaf of the sphere spectrum is the $I$-adic filtration. Therefore there is no chance for the induced filtration on homotopy sheaves to be separated if this is not the case for the $I$-adic filtration;
\item as observed by Kriz \cite[Remark 1.1]{levine2013convergence} this shows that Voevodsky's original conjecture is false for the field $\mathbb{R}$; the point here is that $\mathbb{R}$ has \emph{infinite cohomological dimension} at the prime $2$ and thus is not $I$-adically complete.
\end{enumerate}
Motivated by these two considerations, it is natural to consider the element $\rho$ which interpolates between the layers of the $I$-adic filtration:
\[
I^1(k) \xrightarrow{\rho} I^2(k) \xrightarrow{\rho} \cdots I^j(k) \xrightarrow{\rho} I^{j+1}(k) \xrightarrow{\rho} \cdots,
\]
where we were led to guess that $\rho$-completion on the spectrum level is closely related to $I$-adic completion on the homotopy sheaves level. As a ``best hope" one could guess that $\rho$-completion would lead to slice convergence. However, as already observed in \cite{levine2013convergence}, the slice convergence property is not stable under various categorical operations. Therefore, in place of a completed statement, we offer Theorem~\ref{thm:conv-intro} where convergence does hold if $E$ is $\rho$-complete in a strong sense: that it has bounded $\rho$-torsion. We note that being $\rho$-complete means that it is an inverse limit of $\rho$-torsion objects. 

Theorem~\ref{thm:convergence} turns out to be good enough for many applications in motivic homotopy theory, beyond this paper. Most notably, it renders the computations of \cite{ormsby-rondigs} valid without having to slice complete. It also gives a new, streamlined proof of Thomason's homotopy limit problem by the first author and Hopkins \cite[Appendix A]{bachmann-hopkins} based on the slice spectral sequence. We may also view Theorem~\ref{thm:conv-intro} as a \emph{derived} version of Levine-Voevodsky's slice conjecture and we refer the interested reader to Remark~\ref{rmk:levine-conjecture} for more details.

\subsection{\'Etale descent}  Back to the main subject of this paper, we now formulate our main \'etale descent results. Any motivic spectrum $E \in \SH(S)$ comes equipped with a canonical map $E \rightarrow L_{\et}E$, the \emph{\'etale localization}, witnessing $L_{\et}E$ as the initial motivic spectrum receiving a map from $E$ and satisfying \'etale hyperdescent; this map is just the unit map of the usual adjunction between $\SH(S)$ and its \'etale local version $\SH_{\et}(S)$. Based on a version of the map $\rho$ over $S$, we prove:

\begin{theorem} \label{thm:main-intro} (see Theorem~\ref{thm:main})
Suppose $S$ is a scheme locally of finite dimension, and let $1 \le m, n \le \infty$.
Assume the following conditions hold.
\begin{enumerate} 
\item $1/\ell \in S$,
\item For every $s \in S$ we have $\vcd_\ell(s) < \infty$, and
\item There exists a \emph{good $\tau$-self map} (in the sense discussed in \S\ref{sec:bott-sph}):
 \[
 \tau: \1/(\ell^n,\rho^m) \to \1/(\ell^n,\rho^m)(r).
 \]
\end{enumerate} 
Then for every $E \in \SH(S)_{\ell,\rho}^\comp$ the map 
\[
E/(\ell^n,\rho^m) \to E/(\ell^n,\rho^m)[\tau^{-1}]_{\ell,\rho}^\comp
\] is an \'etale localization.
\end{theorem}

We note that if $m, n < \infty$, then the $(\ell,\rho)$-completion is unnecessary; see Remark~\ref{rmk:p-completion-unnec}. A summary of the good $\tau$-self maps that we managed to construct can be found in \S\ref{subsec:summary} and we will discuss this point further below. We point out some specializations of Theorem~\ref{thm:main-intro} at different primes powers.

\begin{enumerate}
\item Suppose that $\ell$ is an odd prime. Then recall that (as discussed in e.g., \S\ref{sec:standard}) for any motivic spectrum $E$, we have a splitting:
\[
E/\ell^n \simeq E^+/\ell^n \vee E^-/\ell^n.
\]
In this situation, \'etale localization only involve the ``$+$'' part of $E/\ell^n$: the map \[ E^+/\ell^n \to E^+/\ell^n[\tau^{-1}] \] is an étale localization.

\item If $\ell =2$ then the situation is more complicated. In the presence of a square root of $-1$ (e.g., over an algebraically closed field), $\rho$-completion is harmless and we get that 
\[ E/2^n \to E/2^n[\tau^{-1}] \] is an étale localization. Otherwise we have to contend ourselves with $\rho$-completion.
\end{enumerate}

\begin{remark}\label{rem:hyp} Thomason's theorem \cite{aktec} in fact proves that Bott inverted algebraic $K$-theory satisfies not only satisfies \'etale descent but also \'etale \emph{hyperdescent} (this has been reproved by \cite{clausen-mathew}). In particular there is no need to hypersheafify in order to obtain a conditionally convergent spectral sequence abutting to \'etale sheafified $K$-theory from \'etale cohomology. This is useful as \'etale sheafification is much less drastic process than hypersheafificaiton. Variations on this theme in motivic homotopy theory are being explored in \cite{bbx} where it is proved that under noetherian, finite dimensional and finite virtual cohomological dimension assumptions the $\infty$-category obtained by formally imposing \'etale hyperdescent, $\SH_{\et}(S)$, coincides with one obtained by formally imposing just \'etale descent. The input, just as in \cite{clausen-mathew}, is a Quillen-Lichtenbaum style result for the motivic sphere spectrum, analogous to Rost--Voevodsky's results for $K$-theory.

Granting this result, the reader may feel free to interpret the ``\'etale localization" appearing in Theorem~\ref{thm:main-intro} as inverting all desuspensions of \v{C}ech nerves of \'etale covers (as opposed to \'etale hypercovers). 
\end{remark}


\subsection{Strategy} Dealing with the prime $2$ is one of the key technical challenges of this paper that we were able to overcome in many cases. 
 We will now explain this point and the overall strategy of the proof of Theorem~\ref{thm:main-intro}. We proceed by examining the slice spectral sequence for the Bott inverted sphere spectrum over a field (Corollary~\ref{cor:sphere-inversion}), following the general strategy in \cite{elso} for the case of algebraic cobordism. An examination of the form of the slices \cite[Theorem 2.16]{1-line} reveals that they are just motivic cohomology and hence, modulo certain cases which are dealt with using the Beilinson--Lichtenbaum conjectures, satisfy \'etale descent after inverting these Bott elements due to Levine's results \cite{levine2000inverting}. However, as already elaborated in \cite[Section 4]{elso}, examining the slice tower of the Bott inverted sphere spectrum requires a delicate analysis --- the crux point is that the process of Bott inversion is a \emph{colimit}, while the slices (or, more accurately, the co-slices) try to approximate the sphere spectrum as a \emph{limit}. At the prime $2$, and for fields with non-finite $2$-cohomological dimension (such as $\R$), the situation is worse: even the slice filtration itself is \emph{not} convergent \cite[Remark]{levine2013convergence}. From this point of view, the situation seems hopeless. 

On the other hand, previous attempts to cope with this infinitary phenomenon to still obtain a Thomason-style descent theorem have found success under the assumption of finite \emph{virtual cohomological dimension}. To our knowledge, the first paper of this sort was written by the third author in \cite{pa-thesis} and later generalized in joint work with Rosenschon \cite{ro}. The case of hermitian $K$-theory was settled in \cite{bkso}.

It is at this spot that we employ a different analysis from \cite{elso}, which also led to substantial improvements even for the case of algebraic cobordism, other oriented theories and also recovers the descent results of the papers in the previous paragraph. In \emph{loc. cit.}, particularly in \cite[Section 6.5]{elso} and \cite[Section 4]{elso}, the last two authors (with Levine and Spitzweck) examined the resulting slice spectral sequences on Bott-inverted algebraic cobordism and on \'etale algebraic cobordism. This relied on subtle convergence results on inverting elements in a spectral sequence and an analysis of the constituent spaces and bonding maps of an \'etale-localized motivic spectra in a range. In this paper, we work \emph{directly} with the slice filtration and prove a convergence result in the form of Lemma~\ref{lemm:key} which cleanly isolates the role of the convergence of the slice filtration. At the prime $2$ and with the weaker assumption that the ambient field $k$ has finite virtual cohomological dimension, it turns out that convergence holds after completion with respect to the map 
\[
\rho
\colon
\1
\rightarrow
\Gm
\]
induced by the unit $-1\in k^{\times}$. This last statement heavily relies on our Theorem~\ref{thm:conv-intro} on slice convergence.


\subsection{Bott elements and multiplicativity of the Moore spectrum} There are two further, related, issues which we address in the paper: (1) construction of a suitable ``spherical" Bott element and (2) the lack of a multiplicative structure on the motivic Moore spectrum. What is at stake in point (1) is the fact that the sphere spectrum is \emph{not} oriented. In previous iterations of Thomason-style descent results like in \cite{aktec}, \cite{jardine}, \cite{elso}, one produces Bott elements out of $n$-th roots of unity present in the ambient scheme/ring/field which is naturally an $n$-torsion element in an appropriate group; see \cite[Section 6.4]{elso} or \cite[Appendix A]{aktec}. For the sphere spectrum, due to so-called ``$\eta$-logarithmic relation" in Milnor-Witt $K$-theory
\[
[ab] = [a] + [b] + \eta[a][b],
\]
a root of unity is $n_{\epsilon}$-torsion (see Proposition~\ref{prop:powers-n-epsilon}), rather than $n$-torsion. 
At odd primes, the right thing to do is to look at the part of the sphere spectrum where the Milnor-Witt number $n_{\epsilon} = n$, i.e., the so-called ``$+$" part. 

At the prime $2$ the story gets more interesting, coupled with the usual complication that the mod-$2$ Moore spectrum \emph{does not} admit a multiplication. Producing a mod-$2$ Bott element is the first time where $\rho$-completion enters the picture: see \S\ref{subsec:quad}. To address the multiplicativity issues around the mod-$2$ Moore spectra, various authors have considered Oka's action of the mod-$4$ Moore spectrum on the mod-$2$ Moore spectrum. This is actually \emph{insufficient} for our reduction steps (see the argument in Theorem~\ref{thm:main}), primarily because we do not know that this module structure satisfies the usual associativity axiom. In the appendix, we use an idea originating in the work of Davis--Lawson and Hopkins, that this action can be made \emph{asymptotically associative} if we are willing to consider the pro-system $\{ \1/2^n \}$; see \S\ref{sec:moore-mult}. 

In any case, \S\ref{subsec:summary} summarizes the Bott elements/self-maps that we were able to construct in motivic homotopy theory. This adds to the ever-growing list of canonical and vastly interesting elements in stable motivic homotopy theory (such as $\rho$ and $\eta$). 


\subsection{Some applications} We now summarize some applications of our main results. These results unveil certain surprising properties of the $\infty$-category (rather, the premotivic functor) $\SH_{\et}$ which were not a priori visible without our main result. 

\begin{theorem} \label{thm:apps} Assumptions as in Theorem \ref{thm:main-intro}, where $m,n<\infty$:
\begin{enumerate}
\item If $S$ is defined over a field containing a primitive $\ell$-th root of unity and satisfying $k^\times/\ell = \{1\}$, then there exists a Bott element $\tau$ in $\1_{\ell}^\comp$ and \[ \SH_\et(S)_\ell^\comp \wequi \SH(S)_\ell^\comp[\tau^{-1}]. \]
\item In general, $(\ell,\rho)$-complete \'etale localization is equivalent to Bousfield localization at the homology theory $\1/(\ell^n, \rho^m)[\tau^{-1}]$.
\item \'Etale localization is smashing on $\SH(S)_{\ell,\rho}^\wedge$.
\item If $f: T \rightarrow S$ is finite type, the base change functor $f^*$ on $\SH(S)_{\ell,\rho}^\wedge$ preserves \'etale local objects so that the \'etale local sphere (in the $(\rho,\ell)$-complete category) is stable under base change.
\end{enumerate}
\end{theorem}

These corollaries are discussed right after Theorem~\ref{thm:main}. The first statement should be considered the ``model statement" --- the limitation comes from the fact that the construction of higher Bott elements requires more multiplicative structure on the $\rho$-complete sphere than what we already know. In light of this, the second and third statement are weaker but also pleasant consequences of our results. In particular, they tell us that on the $(\rho,\ell)$-complete categories the inclusion of \'etale-local objects preserves colimits. Lastly, base change results in motivic homotopy theory are rewarding but often hard to come by --- our results proves this by way of knowing that the $\tau$-self maps are manifestly stable under base change.

\subsection{Overview of \'etale motivic cohomology theories} This paper completes, in many cases, our structural understanding of how \'etale motivic theories behave; we now sketch this. Suppose that $\mathcal{T}$ is a premotivic functor in the sense of \cite{triangulated-mixed-motives} or a functor satisfying Ayoub's axioms as in \cite[1.4.1]{ayoubthesis}. The most prominent examples are $\mathcal{T} = \SH, \DM\footnote{For this section, we take this in the sense of \cite[Chapter 11]{triangulated-mixed-motives}}$ or $\Mod_{E}$ where $E$ is a (highly structured) motivic ring spectrum such as $\MGL$ which is defined over $\Z$. We can also consider the \'etale local version of $\mathcal{T}$, which we denote by $\mathcal{T}_{\et}$ and comes equipped with a premotivic adjunction
\[
L_{\et}: \mathcal{T} \rightleftarrows \mathcal{T}_{\et}: i_{\et}.
\]

Consider the following categorified version of the arithemetic fracture square, which is a cartesian square of stable $\infty$-categories:
\[
\begin{tikzcd}
\mathcal{T}_{\et}' \ar{r} \ar{d} & \prod_\ell (\mathcal{T}_{\et})^{\wedge}_{\ell} \ar{d}\\
\mathcal{T}_{\et,\mathbb{Q}} \ar{r} & (\prod_\ell (\mathcal{T}_{\et})^{\wedge}_{\ell})_{\mathbb{Q}}.
\end{tikzcd}
\]
We have a natural induced functor $\mathcal{T}_{\et} \rightarrow \mathcal{T}'$ which is fully faithful but not necessarily essentially surjective. Nonetheless, this does mean that the values of invariants evaluated on schemes can be computed in $\mathcal{T}'_{\et}$, and thus an understanding of $\mathcal{T}'_{\et}$ gives us a lot of new information on $\mathcal{T}_{\et}$ as we now explain.

The rational part $\mathcal{T}_{\et,\Q}$ coincides with $\mathcal{T}_{\mathbb{Q}}$ in many cases. In fact, a result of Cisinski-D\'eglise (summarized in, say, \cite[Theorem 12.2]{elso}) furnishes equivalences
\[
\SH_{\et}(S)_{\Q} \simeq \SH(S)^+_{\Q} \simeq \DM(S)_{\Q},
\]
for any Noetherian, geometrically unibranch scheme $S$. Furthermore, as discussed in Remark~\ref{rem:recoll}, the \emph{prime-at-$\ell$} of $\mathcal{T}_{\et}$ part is always zero. Hence, what remain are the $\ell$-complete parts of $\mathcal{T}_{\et}$ where $\ell$ is coprime to the residual characteristics. In this situation, we seek two types of theorems concerning the induced adjunction
\[
L_{\et}: \mathcal{T}^\comp_\ell \rightleftarrows (\mathcal{T}_{\et})^\comp_\ell: i_{\et}.
\]

\begin{enumerate}
\item[Suslin-style rigidity] If $\ell$ is coprime to the residual characteristics, then $(\mathcal{T}_{\et})^\comp_\ell$ is described as a certain category of sheaves over the \emph{small} \'etale site.
\item[Thomason-style descent] If $\ell$ is coprime to the residual characteristics, then the endofunctor $i_{\et}L_{\et}$ is computed explicitly by an inversion of a ``Bott-element" $\tau$.
\end{enumerate}
 
In conjunction, these two results give us access to the values of the cohomology theories represented in $\mathcal{T}_{\et}$. One way to make this concrete, at least in cases where we have ``Postnikov completeness" of the \'etale site of a scheme $X$ (see \cite[Section 2]{clausen-mathew} for a modern discussion when this happens) we obtain a conditionally convergent spectral sequence:
\[
H^p_{\mathrm{et}}(X, \ul{\pi}^{\et}_{q,w}(E)) \Rightarrow [X(w)[q-p] , E[\tau^{-1}]] = (E[\tau^{-1}])^{p-q,-w}(X)
\]
The input of this spectral sequence is obtained as consequence of Suslin-rigidity; we note that it consists of usual \'etale cohomology groups with coefficients in a (torsion/$\ell$-complete) sheaf of abelian groups on the small \'etale site of $X$. The target is obtained as a consequence of Thomason-style descent. This philosophy was already known to Thomason at the beginning in \cite{aktec}. This should be contrasted with the slice spectral sequence, available in the Nisnevich/Zariski setting, where the input consists of motivic cohomology groups which are largely unknown.

We briefly recall Suslin-style rigidity. In \cite{bachmann-SHet,bachmann-SHet2}, the first author has established (again, in many cases) Suslin-style rigidity for $\mathcal{T} =\SH$. Earlier, analogous results were established by R{\"o}ndigs-{\O}stv{\ae}r \cite{rigidity-in-motivic-homotopy-theory} and for $\mathcal{T} = \DM$ by Ayoub \cite{ayoub2014realisation} and Cisinski-D\'eglise \cite{etalemotives}, building on the case of fields where we have Suslin's eponymous result (see \cite[Theorem 7.20]{mvw} for an exposition). 

In this light, what one needs to understand is the precise gap between $\mathcal{T}$ and $\mathcal{T}_{\et}$. For $\mathcal{T} = \SH$ (and up to certain $\rho$-completions), this is exactly the $\tau$-complete category, which remains mysterious.

\subsection{Terminology and notation}
We freely use the language of $\infty$-categories, as set out in \cite{HTT,HA}.
\subsubsection{Motivic homotopy theory}
\begin{itemize}
\item We denote by $\SH_\et(S)$ the localization of $\SH(S)$ at the étale hypercovers and by $\SH(S_\et)$ the stabilization of the \emph{small} hypercomplete étale $\infty$-topos of $S$.
  We denote by $L_\et$ the (various) étale hyperlocalization functors.
  We call a map $E \to F$ an étale localization if $L_\et E \wequi L_\et F \wequi F$.
\item We denote by $E_\ell^\comp = \lim_n E/\ell^n$ (or sometimes $E/\ell^\infty$) the $\ell$-completion of a spectrum (see e.g. \cite[Example 2.3]{bachmann-SHet}; we write $\SH(S)_\ell^\wedge$ for the category of $\ell$-complete motivic spectra.
\item We denote by $\SH(k)_{\ge 0}, \SH(k)_{\le 0}$ the homotopy $t$-structure on $\SH(k)$ defined in \cite[Section 5.2]{morel-trieste}.
\item We put $\1(1) = \Gm[-1]$ and $\1(n) = \1(1)^{\wedge n}$.
  For $E \in \SH(S)$ we put $E(n) = E \wedge \1(n)$.
  \item As is standard we write $T = \A^1/\A^1 - 0$ for the Tate object.
\item We denote by $H\Z$ Spitzweck's motivic cohomology spectrum \cite{spitzweck2012commutative} and write $\DM(S)$ for the $\infty$-category of modules over $H\Z$ \cite{RondigsModules}.
\end{itemize}

\subsubsection{Field theory} We will adopt the following terminology concerning field theory. 
Let $k$ be a field, $k^{\mathrm{sep}}$ a separable closure of $k$ and $G_k:=\mathrm{Gal}(k^{\mathrm{sep}}/k)$ the absolute Galois group. 
Moreover, we write:
\begin{itemize}
\item $\cd_\ell(k)$ for the \emph{$\ell$-cohomological dimension} of $k$ in the sense of \cite[Section 3.1]{serre-gal}.
\item $\vcd_2(k):= \cd_2(k[\sqrt{-1}])$ for the \emph{virtual $2$-cohomological dimension}.
\item More generally, for an integer $s$ we put \[ \vcd_s(k) := \max\{\vcd_\ell(k) \mid p|s\}, \text{ where } \vcd_\ell(k) := \cd_\ell(k(\sqrt{-1})). \]
Note that $\vcd_\ell(k) = \cd_\ell(k)$ unless possibly if $\ell=2$ \cite[Proposition II.10']{serre-gal}.
\item The \emph{cohomological dimension} $\cd(k)$ of $k$ is defined as $\cd(k) = \sup_\ell\{ \cd_\ell(k)\}$.
\end{itemize}

\subsection{Acknowledgements}
We would like to thank Joseph Ayoub, Jeremy Hahn, Mike Hopkins, Tyler Lawson, Marc Levine, Denis Nardin, and Markus Spitzweck for helpful discussions. We also thank an anonymous referee for comments on a previous draft.
The authors thank the Isaac Newton Institute for Mathematical Sciences for support and hospitality during the programme $K$-theory, algebraic cycles and motivic homotopy theory.
This work was supported by EPSRC Grant Number EP/R014604/1 and the RCN Frontier Research Group Project no. 250399 ``Motivic Hopf Equations" and no. 312472 ``Equations in Motivic Homotopy Theory." 
Part of this work was carried out while Elmanto was a postdoc at the Center for Symmetry and Deformation at the University of Copenhagen, 
which also supported Bachmann's visit in 2019. 
We are grateful for the Center's support.
{\O}stv{\ae}r was partially supported by the Humboldt Foundation, 
Professor Ingerid Dal and sister Ulrikke Greve Dals prize for excellent research in the humanities, 
and a Guest Professorship under the auspices of The Radbound Excellence Initiative.

\section{Preliminaries}
\subsection{Endomorphisms of the motivic sphere} \label{sec:standard}
We denote by \[ \lra{-1}: \1 \to \1 \in \SH(S) \] the map corresponding to the switch map $\P^1 \wedge \P^1 \to \P^1 \wedge \P^1$.
Clearly we have $\lra{-1}^2 = 1$.
It follows that if $E \in \SH(S)$ is such that $E \xrightarrow{2} E$ is an equivalence, then there is a canonical decomposition \[ E \wequi E^+ \vee E^- \] which is characterised by the fact that $\lra{-1}$ acts as the identity $\id$ on $E^+$ and as $-\id$ on $E^-$ \cite[\S16.2.1]{triangulated-mixed-motives}.
We denote by \[ \SH(S)[1/2]^+, \SH(S)[1/2]^- \subset \SH(S)[1/2] \] the full subcategories of those objects with $E \wequi E^{\pm}$. For $n \in \N$ we define the \emph{$n$-th Milnor-Witt number} as \[ n_\epsilon = \sum_{i = 1}^n \lra{(-1)^{i-1}} \in [\1, \1]_{\SH(S)}, \] where $\lra{1} := 1$.
We will make use of the following result.
\begin{proposition} \label{prop:powers-n-epsilon} Suppose that $S$ is a base scheme.
\begin{enumerate}
\item If $-1$ is a square on $S$, then $\lra{-1} = 1$.
\item Denote by $p_n: \Gm \to \Gm$ the map corresponding to $x \mapsto x^n$ (i.e., the $n$-th power map), and by $\bar{p}_n: \1 \to \1$ the desuspended map.
  Then $\bar{p}_n = n_\epsilon$.
\end{enumerate}
\end{proposition}
\begin{proof} 
(1) We claim the following more general statement: for $a \in \mathcal{O}^{\times}(S)$, denote by $\langle a \rangle: \1 \rightarrow \1$ the endomorphism in $\SH(S)$ induced by the endomorphism of $S$-schemes $\P^1 \rightarrow \P^1, [x:y] \mapsto [ax:y]$. Then we have
\[
\langle a^2 \rangle = 1.
\]

To prove the claim, we first note that the map $(x:y) \mapsto (a^2x:y)$ is equal to $(x:y) \mapsto (ax:a^{-1}y)$. We have a map $\GL_2(S) \to \Map(\P^1, \P^1)$. It is thus enough to connect the matrix $A = \begin{bmatrix}a^{-1} & 0 \\ 0 & a\end{bmatrix}$ to the identity matrix in $\GL_2(S)$ via $\A^1$-paths. By ``Whitehead's Lemma'', $A$ is a product of elementary matrices: \[ \begin{bmatrix}a^{-1} & 0 \\ 0 & a\end{bmatrix} = \begin{bmatrix}1 & 1/a \\ 0 & 1\end{bmatrix} \begin{bmatrix}1 & 0 \\ 1-a & 1\end{bmatrix}\begin{bmatrix}1 & -1 \\ 0 & 1\end{bmatrix}\begin{bmatrix}1 & 0 \\ 1-a^{-1} & 1\end{bmatrix}. \]
Since the space of elementary matrices is $\A^1$-path connected, the result follows.

(2) For this proof, we will need some rudiments of the theory of framed correspondences in the sense of \cite{EHKSY, garkusha2014framed}. Under the equivalence $\Sigma\Gm \simeq \P^1$, the map $p_n$ suspends to the map $q_n: \P^1 \to \P^1, (x:y) \mapsto (x^n:y^n)$; hence it suffices to prove the claim for this map. Further note that the projection $c: \P^1 \to \P^1/\P^1 \setminus 0 \wequi T$ is an equivalence, $\P^1 \setminus 0 \wequi \A^1$ being contractible. It follows that it suffices to determine the map 
\[
c \circ q_n: \P^1 \rightarrow T
\] which is precisely the map induced via Voevodsky's Lemma \cite[Corollary A.1.7]{EHKSY} from the equationally framed transfer with support $0\in \A^1$ and framing $t^n$ (in the sense of \cite[Definition 2.1.1]{EHKSY}). It follows from the main result of \cite{EHKSY2} (in particular, the agreement results \cite[Theorems 3.2.11, 3.3.6]{EHKSY2}) that this is stably the same as the map $\1 \to \1$ induced by the corresponding framed correspondence via the reconstruction theorem of \cite[Theorem 3.5.11]{EHKSY} and its generalization over any base \cite[Theorem 16]{framed-loc}. The result now follows from \cite[Proposition B.1.4]{EHKSY}, provided that their $n_\epsilon$ is the same as ours. Their $\lra{a}$ comes from the canonical action of $\scr O^\times(S) \to \Omega K(S)$ on framed correspondences, which coincides with the canonical action via fundamental classes by \cite[Theorem 3.2.11]{EHKSY2}, which coincides with our definition essentially by construction.
\end{proof}

\begin{remark} Given $u \in \scr O^\times(S)$ we let $[u]$ denote the resulting map $\1 \to \Gm$. 
One may show that \[ \lra{-1} = 1 + \eta[-1], \] where $\eta: \Gm \to \1$ is the geometric Hopf map \cite[\S6.2]{morel-trieste}.
\end{remark}

\subsection{$\rho$-completion}
In this section, we discuss $\rho$-completed motivic spectra in view of certain convergence results that we will use later; we use \cite[Section 2.2]{mnn} as a reference for the formalism of complete objects and completions; see also \cite{rigidity-in-motivic-homotopy-theory} for a previous reference in the motivic context. We put \[ \rho := [-1]: \1 \to \Gm \in \SH(S). \] 
By abuse of notation, we also denote the map $\Gmp{-1} \to \1$ (obtained by smashing $\rho$ with $\id_{\Gmp{-1}}$) by $\rho$.
We write $\1/\rho$ for the cofiber of $\rho: \Gmp{-1} \to \1$.
Recall that a morphism $\alpha: E \to F \in \SH(S)$ is a \emph{$\rho$-equivalence} if $\alpha \wedge \1/\rho: E/\rho \to F/\rho$ is an equivalence. We denote by \[ \SH(S)_\rho^\comp \] the localization of $\SH(S)$ at the $\rho$-equivalences. Recall that $E \in \SH(S)_\rho^\comp$, i.e., it is \emph{$\rho$-complete}, if and only if for all $F$ such that $F \otimes \1/\rho \simeq 0$, the space $\Map(F, E)$ is contractible.

From general principles the localization $\SH(S) \to \SH(S)_\rho^\comp$ has a fully faithful right adjoint, and the composite \emph{localization functor} is given by \cite[Formula (2.22)]{mnn} \[ E \mapsto E_\rho^\comp := \lim_n E/\rho^n. \]

\begin{remark}
Functors between premotivic categories that preserve cofibers and smashing with $\Gm$ preserve $\rho$-equivalences. Since $\Gm$ is invertible the latter condition is in particular satisfied by symmetric monoidal functors and their adjoints. It follows that the motivic base change functors $f^*, f_\#$, and $f_*$ preserve $\rho$-equivalences.
\end{remark}
\begin{warning}
On the other hand, the functors related to the slice filtration, like $f_n$ and $s_n$, do not interact well with $\rho$-equivalences.
See for example Remark \ref{rmk:fM-rho}.
\end{warning}

In order to streamline the exposition, in what follows we will extensively use the category $\SH(S)_\rho^\comp$.
In many cases, however, this has very little effect:

\begin{remark} \label{rmk:rho-completion}
\begin{enumerate}
\item Suppose $RS = \emptyset$\footnote{For a scheme $X$, we denote by $RX$ the set of pairs $(x, \alpha)$ with $x \in X$ and $\alpha$ an ordering of $k(x)$.}, 
i.e., $-1$ is a sum of squares locally in $\scr O_S$ \cite[Theorem 4.3.7]{bochnak2013real}.
  Then $\rho \in \pi_0(\1)_*$ is nilpotent, as follows, for example, from the main result of \cite{bachmann-real-etale}.
  It follows that \[ \SH(S)_\rho^\comp = \SH(S). \]
\item Similarly, over a general base, $\rho$ is nilpotent on $\1[1/2]^+$ and a unit on $\1[1/2]^-$, see \cite[Lemma 39]{bachmann-real-etale}.
  It follows that \[ \SH(S)[1/2]_\rho^\comp = \SH(S)[1/2]^+. \]
\end{enumerate}
\end{remark}

\begin{example} \label{ex:HZ-rho-complete}
Let $S$ be essentially smooth over a Dedekind domain.
Then $H\Z \in \SH(S)$ is $\rho$-complete.
Indeed it suffices to show that for $X \in \Sm_S$, $i, j \in \Z$, we have \[ [\Sigma^{i,j} X_+, H\Z \wedge \Gmp{-d}] = 0 \] for $d$ sufficiently large.
This follows from the vanishing of motivic cohomology in negative weights over such bases \cite[Corollary 7.19]{spitzweck2012commutative}.
\end{example}

\subsection{Virtual étale cohomological dimension} \label{subsec:virtual-coh-dim} We will need the following result about motivic cohomology in large degrees.

\begin{lemma} \label{lemm:vcd-iso}
For $m > \vcd_\ell(k)$ and $a \in \Z$ the map \[ H^m(k, \Z/\ell(a)) \xrightarrow{[-1]} H^{m+1}(k, \Z/\ell(a+1)) \] is an isomorphism.
\end{lemma}
\begin{proof}
If $a < m$ then both groups vanish by \cite[Lemma 3.2(2)]{MR1744945} (\cite[Lemma 5.2]{mvw}), so there is nothing to prove.
Thus let $a \ge m$, 
so that \[ H^m(k, \Z/\ell(a)) \wequi H^m_\et(k, \mu_\ell^{\otimes a}) \] by the Beilinson-Lichtenbaum conjecture \cite[Theorem 6.17]{voevodsky-BK} (\cite[Theorem 10.2]{mvw}).
If $\ell$ is odd then $\cd_\ell(k) = \vcd_\ell(k)$ and again both groups are zero.
If $\ell=2$ then $\mu_\ell^{\otimes a} = \Z/2$ and the claim follows from the Gysin sequence (see e.g., Lemma \ref{lemm:vcd-vanishing}).
\end{proof}

\section{Bott-inverted spheres}\label{sec:bott-sph}
\subsection{\texorpdfstring{$\tau$}{tau}-self maps}
\begin{definition} \label{def:tau-self-map}
Let $\ell$ be a prime with $\ell \in \mathcal{O}^\times(S)$ and $1 \le m, n \le \infty$.
If $m, n < \infty$ denote by $\1/(\ell^n, \rho^m)$ the obvious cofiber; if $n = \infty$ and $m<\infty$ we set $\1/(\ell^\infty, \rho^m) = (\1/\rho^m)_\ell^\comp$, 
and similarly if $m = \infty$.
Suppose given a map \[ \tilde\tau: \1/(\ell^n, \rho^m) \to \1/(\ell^n, \rho^m)(r). \]
\begin{enumerate}
\item We write \[ \1/(\ell^n,\rho^m)[\tilde\tau^{-1}] = \colim \left[ \1/(\ell^n,\rho^m) \xrightarrow{\tilde\tau} \1/(\ell^n,\rho^m)(r) \xrightarrow{\tilde\tau(r)}\1/(\ell^n,\rho^m)(2r) \to \dots \right]. \]
  More generally, for $E \in \SH(S)$ we set \[ E/(\ell^n, \rho^m)[\tilde\tau^{-1}] = E \wedge \1/(\ell^n,\rho^m)[\tilde\tau^{-1}]. \]

\item We call $\tilde\tau$ a $\tau$-self map if for every map $x \to S$, where $x$ is the spectrum of a field, the map \[ H\Z_x/(\ell^n, \rho^m) \to H\Z_x/(\ell^n, \rho^m)[\tilde\tau^{-1}]_{\ell,\rho}^\wedge \in \SH(x)_{\ell,\rho}^\comp \] is an étale localization.

\item We call $\tilde\tau$ a good $\tau$-self map if it is a $\tau$-self map and, for every $E, F \in \SH(S)$, every morphism $\alpha: E \to F/(\ell^n,\rho^m)[\tilde\tau^{-1}]_{\ell,\rho}^\comp$ factors through $E \to E/(\ell^{n'},\rho^{m'})[\tilde\tau'^{-1}]$, where $\tilde\tau'$ is another $\tau$-self map.
\end{enumerate}
\end{definition}

\begin{remark} \label{rmk:p-completion-unnec}
Unless $m=\infty$ or $n=\infty$, $E/(\ell^n, \rho^m)[\tilde\tau^{-1}]$ is already $(\rho, \ell)$-complete.
\NB{Call $E$ $a$-torsion if $a: E \to E$ is the zero map. 
Then sums of $a$-torsion spectra are $a$-torsion, and an extension of $a$-torsion spectra is $a^2$-torsion. Hence directed colimits of $a$-torsion spectra are $a^2$-torsion.}
\end{remark}

\begin{lemma} \label{lemm:bott-element-etale-invertible}
Let $\tau: \1/(\ell^n, \rho^m) \to \1/(\ell^n, \rho^m)(r)$ be a $\tau$-self map.
Then $L_\et(\tau)_{\ell,\rho}^\comp$ is an equivalence.
\end{lemma}
\begin{proof}
Suppose that $n=\infty$.
The map $L_\et(\tau)_{\ell,\rho}^\comp$ is an equivalence if and only if $L_\et(\tau/\ell)_{\ell,\rho}^\comp$ is (by $\ell$-completeness).
Since $\tau/\ell$ is also a $\tau$-self map, replacing $\tau$ by $\tau/\ell$ we may assume that $n<\infty$.
Similarly we may assume that $m<\infty$.

By \cite[Theorem 3.1]{bachmann-SHet2} we have $\SH_\et(S)_\ell^\wedge \wequi \SH(S_\et^\wedge)_\ell^\comp$.
This implies that pulling back along maps of the form $Spec(k) \to S$, where $k$ is separably closed, is a conservative family for $\SH_\et(S)_\ell^\comp$.
We may thus replace $S$ by the spectrum of a separably closed field $k$.
\NB{The only reason for this is to know that isos are detected in homology. Easy to see ab initio, too.}

Since $L_\et \1/(\ell^n,\rho^m) \in \SH_{\ge -1}$, in order to show that $L_\et(\tau)_{\ell,\rho}^\comp$ is an equivalence, it suffices to show this in homology.
In other words, we need to show that the map \[
H\Z/(\ell^n,\rho^m) \xrightarrow{\tau} H\Z/(\ell^n,\rho^m)(r)
\] is an étale equivalence.
This is true since by part (2) of Definition~\ref{def:tau-self-map}, we demand that $H\Z/(\ell^n,\rho^m) \to H\Z/(\ell^n,\rho^m)[\tau^{-1}]_{\ell,\rho}^\comp$ is an étale localization.
\end{proof}

\subsection{Cohomological Bott elements}
We shall in the next subsection construct $\tau$-self maps as multiplication by suitable elements.
In preparation, we study the analogous question for $H\Z$.

\begin{definition} \label{def:cohomological-bott-element}
Let $S = Spec(k)$, $k$ a field.
Let $1 \le m, n \le \infty$.
By a \emph{cohomological Bott element} we mean an element $\tau \in \pi_{0,-r} H\Z/(\ell^n, \rho^m)$ for some $r>0$ such that for every field $l/k$ and every choice of multiplication on $H\Z/(\ell^n, \rho^m)$ under the standard multiplication on $H\Z_\rho^\wedge/\ell^n$ (in the sense of \S\ref{sec:moore-mult}), the map \[ H\Z/(\ell^n,\rho^m)|_l \to H\Z/(\ell^n,\rho^m)|_l[\tau^{-1}]_{\ell,\rho}^\wedge \] is an étale localization.

Here, 
$H\Z/(\ell^n,\rho^m)[\tau^{-1}]$ denotes the mapping telescope of the self map given by multiplication by $\tau$.
\end{definition}

\begin{remark}
In light of Example \ref{ex:HZ-rho-complete}, $H\Z/\ell^m$ is already $\rho$-complete, so we shall suppress the additional completion from now on.
\end{remark}

\begin{remark} \label{rmk:product-bott-elts}
If $\tau_1, \tau_2$ are cohomological Bott elements, then so is $\tau_1 \tau_2$.
Indeed by cofinality, $E[(\tau_1\tau_2)^{-1}]$ (where $E=H\Z/(\ell^n, \rho^m)$) can be computed as a colimit over $\N \times \N$, with horizontal maps given by multiplication by $\tau_1$ and vertical maps by $\tau_2$.
Computing the horizontal colimit first, it suffices to show that multiplication by $\tau_2$ induces an equivalence on $E[\tau_1^{-1}]$; this follows from the assumption that $\tau_1$ and $\tau_2$ are both cohomological Bott elements.
\end{remark}

\begin{lemma} \label{lemm:bott:powers-reduction}
Let $\tau_1 \in \pi_{0,-r} H\Z/\ell$ and $\tau_2 \in \pi_{0,-r'} H\Z/\ell^{n}$ be elements such that $\tau_2$ reduces to a power of $\tau_1$ modulo $\ell$ (here $n=\infty$ is allowed).
\begin{enumerate}
\item $\tau_1$ is a cohomological Bott element if and only if $\tau_2$ is a cohomological Bott element.
\item If $\tau_2$ is a cohomological Bott element, then its image  $\tau_3$ in $\pi_{0,-r'}H\Z/(\ell^n,\rho^m)$ is a cohomological Bott element.
\end{enumerate}
\end{lemma}
\begin{proof}
(1) Suppose first that $n < \infty$.
We have a cofiber sequence $H\Z/\ell^n \xrightarrow{\ell} H\Z/\ell^n \to H\Z/\ell \oplus \Sigma H\Z/\ell^{n-1}$ in $\SH(k)$.
Thus if $H\Z/\ell^n[\tau_2^{-1}]$ is étale local then so is $H\Z/\ell[\tau_2^{-1}] \wequi H\Z/\ell[\tau_1^{-1}]$.
Moreover, the converse holds if also $H\Z/\ell^{n-1}[\tau_2^{-1}]$ is étale local; this will hold by induction on $n$.
Finally suppose $n=\infty$.
Then $H\Z/\ell[\tau_1^{-1}] \wequi H\Z_\ell^\comp[\tau_2^{-1}]_\ell^\comp/\ell$, so $\tau_1$ is a cohomological Bott element if $\tau_2$ is.
Conversely, if $\tau_1$ is a cohomological Bott element then $H\Z_\ell^\comp[\tau_2^{-1}]_\ell^\comp \wequi \lim_n H\Z/\ell^n[\tau^{-1}]$ is a limit of étale local spectra, by what we already established, so it is étale local.
In other words, $\tau_2$ is a cohomological Bott element.

(2) Write $\tau_3$ for the image.
Then for any choice of compatible multiplication, the map $H\Z/(\ell^n,\rho^m) \to H\Z/(\ell^n,\rho^m)[\tau_3^{-1}]$ is the cofiber of multiplication by $\rho^m$ on the map $H\Z/\ell^n \to H\Z/\ell^n[\tau_2^{-1}]$.
The latter is an étale localization by (1), and hence so is the former.

This concludes the proof.
\end{proof}

\begin{lemma} \label{lemm:bott-descent}
Let $l/k$ be a finite separable extension of degree coprime to $\ell$, and $\tau \in \pi_{0,-r} H\Z/\ell^n$.
Then $\tau$ is a cohomological Bott element if and only if $\tau|_l$ is a cohomological Bott element.
\end{lemma}
\begin{proof}
It suffices to show that $H\Z/\ell^n[\tau^{-1}]$ is a summand of $(l/k)_* H\Z/\ell^n|_l[\tau|_l^{-1}]$.
This is clear by the existence of transfers in $\DM(k)$ (see e.g., Corollary \ref{cor:galois-descent}).
\NB{this is total overkill...}).
\end{proof}

\begin{lemma} \label{lemm:bott-zeta}
Let $\zeta \in k$ be a primitive $\ell^n$-th root of unity and $\tau \in \pi_{0,-1} H\Z/\ell^n$.
Suppose that $\beta(\tau) = [\zeta]$, where \[ \beta: H^{0,1}(k, \Z/\ell^n) \to H^{1,1}(k, \Z) \] is the integral Bockstein and $[\zeta] \in H^{1,1}(k, \Z) \wequi k^\times$ is the element corresponding to $\zeta$. Then $\tau$ is a cohomological Bott element.
\end{lemma}
\begin{proof}
Under the equivalence $\Z(1) \wequi \Gm$ \cite[Lemma 3.2(1)]{MR1744945} (\cite[Theorem 4.1]{mvw}), 
the Bockstein corresponds to the inclusion $\ul{\pi}_0(\Z/\ell^n(1)) \wequi \mu_{\ell^n} \hookrightarrow \Gm$.
It follows that $L_\et(\Z/\ell^n(1)) \wequi \mu_{\ell^n}$ and that $\tau$ defines an equivalence $\Z/\ell^n \wequi L_\et(\Z/\ell^n(1))$; in particular $\tau$ is an étale local equivalence.
It remains to show that for $X \in \Sm_k$ the induced map \[ H^{**}(X, \Z/\ell^n)[\tau^{-1}] = \colim_r H^{*,*+r}(X, \Z/\ell^n) \to H^{**}_\et(X, \Z/\ell^n) \] is an isomorphism.
This is immediate from the solution of the Beilinson-Lichtenbaum conjecture \cite[Theorem 6.17]{voevodsky-BK}.
\end{proof}

\subsection{Spherical Bott elements}
\begin{definition} \label{def:spherical-bott-element}
Let $1 \le m, n \le \infty$, $S$ a scheme.
By a \emph{spherical Bott element} we mean an element $\tau \in \pi_{0,-r}(\1/(\ell^n, \rho^m))$ such that for every map $x \to S$ where $x$ is the spectrum of a field, the induced element $H\Z \wedge \tau|_x \in \pi_{0,-r}(H\Z_x/(\ell^n,\rho^m))$ is a cohomological Bott element (see Definition \ref{def:cohomological-bott-element}).
\end{definition}

\begin{remark} \label{rmk:spherical-bott-base-change}
Spherical Bott elements are stable under base change, essentially by definition.
\end{remark}

\begin{lemma} \label{lemm:tau-self-map-construct}
Choose a multiplication\footnote{The existence of this structure will be extensively discussed in Appendix~\ref{sec:oka}.} on $\1/\ell^n \in \SH$ and one on $\1_\ell^\comp/\rho^m \in \SH(S)$.
Let $n' \le n$ and choose a $\1/\ell^n$-module structure $m$ on $\1/\ell^{n'}$.
Let $\tau \in \pi_{0,-r} \1/(\ell^n, \rho^m)$ be a spherical Bott element.

Then the composite \[ \tilde\tau: \1/(\ell^{n'},\rho^m) \xrightarrow{\tau \wedge \id} \Sigma^{0,r} \1/(\ell^{n},\rho^m) \wedge \1/(\ell^{n'},\rho^m)  \xrightarrow{m} \Sigma^{0,r} \1/(\ell^{n'},\rho^m) \] is a $\tau$-self map.
If the multiplication is chosen to be homotopy commutative, and the module structure homotopy associative, then the $\tau$-self map is good.
\end{lemma}
\begin{proof}
We need to check that forming the mapping telescope of $\tilde\tau \wedge H\Z|_k$ is an étale localization, for every field $k$.
This is true basically by definition, as soon as we know that the multiplication on $H\Z/\ell^n$ induced by the one on $\1/\ell^n$ is the standard one.
This is indeed the case, by Corollary \ref{cor:multn-HZ-correct}.

Goodness of the $\tau$-self map follows from Corollary \ref{cor:factor-through-free}.
\end{proof}

\section{Construction of spherical Bott elements} \label{sec:construction}

\subsection{Construction via roots of unity} \label{sec:bott-primitive}
\begin{lemma} \label{lem:bott-exists}
Let $\ell^n \ne 2$ and suppose $S$ contains a primitive $\ell^n$-th root of unity $\zeta$.
Then there exists a spherical Bott element mod $(\ell^n, \rho^\infty)$.
\end{lemma}
\begin{proof}
Note that either $\ell$ is odd or $-1$ is a square on $S$.
In both cases, using the notation $n_\epsilon$ from \S\ref{sec:standard}, we find that 
\[ \ell^n_\epsilon = \ell^n: \1_{\ell,\rho}^\comp \to \1_{\ell,\rho}^\comp.\] 
By construction we have a cofibration sequence \[ \1_{\ell,\rho}^\comp/\ell^n(1) \xrightarrow{\beta} \Gm_{\ell,\rho}^\comp \xrightarrow{\ell^n} \Gm_{\ell,\rho}^\comp, \] whence by Lemma \ref{lemm:bott-zeta} a spherical Bott element exists as soon as $\ell^n[\zeta]_{\rho,\ell}^\comp = 0$.
This is immediate from Proposition \ref{prop:powers-n-epsilon}.
\end{proof}

\subsection{Lifting to higher $\ell$-powers} \label{sec:bott-power}
Given an object $E$ in a triangulated symmetric monoidal category and a map $m: E \otimes E \to E$, by a \emph{derivation} $\delta: E \to E[1]$ we mean a map such that $\delta \circ m = m \circ (\id \otimes \delta + \delta \otimes \id)$.
\begin{lemma} \label{lemm:power-derivation}
Let $\scr C$ be a symmetric monoidal triangulated category. Suppose $E \in \scr C$ is a non-commutative, non-associative and non-unital algebra, and let $\delta: E \to E[1]$ be a derivation.
Then for $X \in \scr C$ and $t: X \to E$ we have \[ \delta(t^n) = \sum_{i = 1}^n t \dots t \delta(t) t \dots t: X^{\otimes n} \to E[1]. \]
Here the sum is over all possible ways of replacing one instance of $t$ in the string $t^n$ by $\delta(t)$, and both sides are associated from the left (i.e., $t^n = t(t(\dots(tt)\dots))$ and so on).
\end{lemma}
\begin{proof}
We write the map $\delta(t^{n+1})$ as \[ X^{\otimes n+1} \wequi X \otimes X^{\otimes n} \xrightarrow{t \otimes t^{\otimes n}} E \otimes E \xrightarrow{\id \otimes t^n} E \otimes E \xrightarrow{m} E \xrightarrow{\delta} E[1]. \]
The definition of a derivation implies that this is the sum of the two maps $t\delta(t^n)$ and $\delta(t)t^n$.
The result thus follows by induction, starting at $n=2$ where we use the definition of a derivation.
\end{proof}

Recall the notion of a regular multiplication from \S\ref{sec:def-set}.
\begin{proposition} \label{prop:bott-power}
Let $q$ be a prime power, fix a regular multiplication on $\Mp{q} \in \SH$ and assume there exists a spherical Bott element mod $(q,\rho^\infty)$.
Then for every $n \ge 1$ there exists a spherical Bott element mod $(q^n,\rho^\infty)$.
\end{proposition}
\begin{proof}
Throughout the proof, all spectra are implicitly $(\rho,q)$-completed.

As a first step, choose regular multiplications on $\Mp{q^n}$ for all $n \ge 1$ such that the sequence of reduction maps \[ \dots \Mp{q^3} \to \Mp{q^2} \to \Mp{q} \] consists of morphisms of homotopy unital ring spectra. This is possible by \cite[Lemma 5]{oka1984multiplications}.
We shall proceed by induction on $n$.
Thus let $\tau$ be a spherical Bott element modulo $q^n$.
We claim that $\tau^q$ lifts to an element modulo $q^{n+1}$; this will be a spherical Bott element by Lemma \ref{lemm:bott:powers-reduction}(1).
To see this, we begin with the cofiber sequence in \cite[(3.3)]{rigidity-in-motivic-homotopy-theory} 
\[ \Mp{q} \to \Mp{q^{n+1}} \to \Mp{q^n} \xrightarrow{\bar\delta} \Mp{q}[1]. \]
We need to show that $\bar\delta(\tau^q) = 0$.
Let $r: \Mp{q^n} \to \Mp{q}$ be the reduction map; then $\bar\delta = r\delta$, where $\delta: \Mp{q^n} \to \Mp{q^n}[1]$ is the coboundary.
By our choice $r$ is multiplicative, and so by Lemma \ref{lemm:power-derivation} we have 
\[ \bar\delta(\tau^q) = r\left[ \delta(\tau)\tau^{q-1} + \tau\delta(\tau)\tau^{q-2} + \dots + \tau^{q-1}\delta(\tau) \right], \] where the sum consists of $q$ terms.

Note that $q = 0$ on $\Mp{q}$, since $\Mp{q}$ has a multiplication.
Suppose first that the multiplication on $\Mp{q^n}$ is associative and commutative.
Then each of the $q$ terms in our sum is the same\footnote{One might be concerned here about an absence of signs. We are given various maps $\1^{\wedge q} \to \1/q[1]$, differing by permutations \emph{of the source only}. But the switch map on $\1 \wedge \1$ is the identity, whence there are no signs.}, so the sum is zero, and we are done.

To finish the proof, we observe that we do not actually need the multiplication to be homotopy associative or commutative.
Since $\delta = u\partial$ (where $u: \1 \to \Mp{q^n}$ is the unit map) we get $\partial \delta = 0$.
This implies that any element commutes with $\delta(\tau)$, and any two elements associate with $\delta(\tau)$ (see Section \ref{sec:oka}).
Thus \[ \delta(\tau)\tau^{m+1} \stackrel{(1)}{=} \delta(\tau)(\tau\tau^m) \stackrel{(2)}{=} (\delta(\tau) \tau)\tau^m \stackrel{(3)}{=} (\tau \delta(\tau))\tau^m \stackrel{(2)}{=} \tau(\delta(\tau)\tau^m), \] where $(1)$ is by definition (everything being associated from the left), $(2)$ is because everything associates with $\delta(\tau)$ and $(3)$ is because everything commutes with $\delta(\tau)$.
This implies by induction that $\delta(\tau)\tau^m$ is independent of the order of multiplication, for any $m$, and so all the terms are the same, as before.
\end{proof}

\subsection{Construction by descent} \label{sec:Bott-descent}
Let $f: S' \to S$ be a morphism of schemes and $\sigma: S' \to S'$ be an automorphism over $S$.
Then $f \sigma = f = f \sigma^{-1}$ and so $\sigma^* f^* \stackrel{\alpha}{\wequi} f^*$, $f_* \sigma^{-1}_* \stackrel{\beta}{\wequi} f_*$.
This provides us with a sequence of equivalences \[ f_*f^* \stackrel{f_*\alpha}{\wequi} f_* \sigma^* f^* \wequi f_* \sigma^{-1}_* f^* \stackrel{\beta f^*}{\wequi} f_*f^*, \]
where we have used that an adjoint of an equivalence is canonically equivalent to the inverse.
Hence for every object $E \in \SH(S)$ we get an automorphism \[ \sigma_E: f_*f^* E \to f_*f^* E. \]
This construction is natural in $E$.

The following result is proved in Appendix \ref{sec:galois}; see Corollary \ref{cor:galois-descent} and the preceding paragraphs.
\begin{proposition} \label{prop:galois-descent}
Let $f:S' \to S$ be a finite Galois cover with group $G$.
\begin{enumerate}
\item The above incoherent construction refines to a coherent action, i.e., a functor 
\[BG \to \SH(S); * \mapsto f_*f^* E.
\]
\item The unit of adjunction $E \to f_*f^*E$ refines to a $G$-equivariant map (for the trivial action of $G$ on $E$).
\item Suppose that $E \in \SH(S)[1/2, 1/|G|]^+$.
  Then $E \to f_*f^*E$ exhibits $E$ as the homotopy fixed points of the $G$-action on $f_*f^* E$, and this limit diagram is preserved by any additive functor.
\end{enumerate}
\end{proposition}

\begin{remark} \label{rmk:descent-summand}
In the situation of Proposition \ref{prop:galois-descent}(3), the map $E \to f_*f^*E$ is a split injection, i.e., $E$ is a summand of $f_*f^* E$.
See Corollary \ref{cor:compute-fixed} for details.
\end{remark}

\begin{corollary} \label{cor:bott-descent}
Let $q=\ell^n$ be an odd prime power, $q \ne 3$.
Assume $1/\ell \in S$ and let $S'$ be obtained from $S$ by adjoining a primitive $\ell$-th root of unity. 
If $S'$ affords a spherical Bott element mod $(q,\rho^\infty)$ then so does $S$.
\end{corollary}
\begin{proof}
We again complete everything implicitly at $(\rho,\ell)$.

If $S=S'$ there is nothing to prove, so assume $S' \ne S$.
Thus $f: S' \rightarrow S$ is a Galois cover with Galois group $G \subset \Z/(\ell-1)$ \cite[Corollary 10.4]{neukirch2013algebraic}.
Consequently \[ \Mp{q} \in \SH(S)[1/2, 1/|G|]^+ \] and so by Proposition \ref{prop:galois-descent} we get $[\1, \Mp{q}(m)]_S = [\1, \Mp{q}(m)]_{S'}^G$.
Let $\tau \in [\1, \Mp{q}(m)]_{S'}$ be a spherical Bott element and put \[ \tau' = \prod_{g \in G} (g \tau) \in [\1, \Mp{q}(|G|m)]_{S'}. \]
Then for $h \in G$ we have \[ h\tau' = \prod_g (hg \tau) = \tau', \] since the multiplication in $\Mp{q}$ is commutative and associative (here we use that $q \ne 3$).
In other words $\tau'$ is fixed by $G$ and so defines an element $\tau_S \in [\1, \Mp{q}(|G|m)]_S$.
By construction, $f^*\tau_S = \prod_g (g^* \tau)$ is a product of spherical Bott elements and hence a spherical Bott element (see Remark \ref{rmk:product-bott-elts}).
By Lemma \ref{lemm:bott-descent} it follows that $\tau_S$ is also a spherical Bott element.
\end{proof}

\subsection{Construction over special fields} \label{subsec:quad}
\subsubsection{Quadratically closed fields}
\begin{proposition} \label{prop:tau-p-closed}
Let $\ell$ be a prime, $S = Spec(k)$.
Assume that $k$ affords a primitive $\ell$-th root of unity,
\NB{so $char(k) \ne \ell$...}, 
and $k^\times/\ell \wequi \{1\}$ (i.e., every element of $k$ admits an $\ell$-th root).
Then there exists a spherical Bott element \[ \tau \in \pi_{0,-1}(\1_{\ell,\rho}^\comp). \]
\end{proposition}
\begin{proof}
Lemma \ref{lem:bott-exists} and its proof show that there exists $\tau_1 \in \pi_{0,-1}(\1_\rho^\comp/q)$, where $q=4$ if $\ell=2$ and $q=\ell$ else.
We shall show by induction that for each $n$: \begin{equation} \label{eq:ind-step} \text{there exists a lifting } \tau_{n+1} \in \pi_{0,-1}(\1_\rho^\comp/q^{n+1}) \text{ of } \tau_n. \tag{$*$} \end{equation}
By the Milnor exact sequence \cite[Proposition VI.2.15]{goerss2009simplicial} there is a surjection \[ \pi_{0,-1}(\1_{\ell,\rho}^\comp) \to \lim_n \pi_{0,-1}(\1_\rho^\comp/q^n); \] 
hence there is a (non-canonical) lift $\tau \in \pi_{0,-1}(\1_{\ell,\rho}^\comp)$.
This will be a spherical Bott element by Lemma \ref{lemm:bott:powers-reduction}(1).

It hence remains to prove \eqref{eq:ind-step}.
The cofiber sequence \[ \1(1)_\rho^\comp/q \to \1(1)_\rho^\comp/q^{n+1} \to \1(1)_\rho^\comp/q^n \to \1[1](1)_\rho^\comp/q \wequi \Gm_\rho^\comp/q \] implies that it is enough to prove the vanishing $\pi_0(\Gm_\rho^\comp/q) = 0$.

Suppose first that $\ell=2$.
Then $\rho$ is nilpotent by Remark~\ref{rmk:rho-completion}(1), so $\Gm_\rho^\comp \wequi \Gm$ and $\pi_0(\Gm_\rho^\comp/q) \wequi K_1^{MW}(k)/4$ \cite[Corollary 6.43]{A1-alg-top}.
Recall that we have the fiber product decomposition \cite[Theorem 5.4]{gsz} (\cite{morel-ideal})
\[ K_1^{MW}(k) \wequi I(k) \times_{I(k)/I^2(k)} K_1^M(k). \]
Since $k$ is quadratically closed, $I(k) = 0$ \cite[Lemma 31.1]{MR2427530} and hence $K_1^{MW}(k) \wequi K_1^M(k)$.
Thus \[ K_1^{MW}(k)/4 \wequi K_1^M(k)/4 \wequi k^\times/4 \wequi \{1\}, \] and we are done (here the last isomorphism again follows from the fact that $k$ is quadratically closed).

Now suppose that $\ell$ is odd.
Then $\Gm_\rho^\comp/q \wequi \Gm[1/2]^+/\ell$ and so \[ \pi_0(\Gm_\rho^\comp/q) \wequi K_1^{MW}(k)[1/2]^+/\ell \wequi K_1^M(k)/\ell \simeq k^{\times}/\ell \simeq \{1 \}. \] Here we have used the fact that for any field $k$, the group $I/I^2(k)[1/2] = 0$ since it is a module over $W/I(k) \wequi \Z/2$ and thus the fiber product decomposition for $K_1^{MW}$ is just a product. 

This concludes the proof.
\end{proof}

\subsubsection{The real numbers} \label{subsec:reals}
If $S = Spec(\R)$, then there exists $\tau_n \in \pi_{0,-r}(\1_2^\comp/\rho^n)$ lifting a power of $\tau \in \pi_{0,-1}H\Z/2$ \cite[Theorem 7.10 and its proof]{behrens2019c_2}.
Moreover $\1_2^\comp/\rho^n$ is an $\scr E_\infty$-ring \cite[Lemma 7.8]{behrens2019c_2}.
Consequently this defines a spherical Bott element modulo $(2^\infty, \rho^n)$ by Lemma \ref{lemm:bott-zeta}(2).

\subsection{Summary of \texorpdfstring{$\tau$}{tau}-self maps} \label{subsec:summary}
Using Lemma \ref{lemm:tau-self-map-construct} and the multiplications and module structures on Moore spectra \cite{oka1984multiplications}, as reviewed (and slightly extended in Corollary~\ref{cor:asymp}) in Appendix \ref{sec:moore-mult}, we find that there is a good $\tau$-self map modulo $(\ell^n, \rho^m)$ as soon as there is a spherical Bott element modulo $(\ell^{n'}, \rho^m)$ for some $n' \ge n$.
In particular we have said elements in the following cases:
\begin{enumerate}
\item $m = \infty$, $\ell$ odd, $n < \infty$ (use Corollary \ref{cor:bott-descent}, Proposition \ref{prop:bott-power}, and Lemma \ref{lem:bott-exists}).
\item $m = \infty$, $\ell=2$, $n<\infty$, $\sqrt{-1} \in S$ (using the same results).
\item $m = \infty$, $n=\infty$, $\ell$ arbitrary, $S$ defined over a field $k$ containing a primitive $\ell$-th root of unity and satisfying $k^\times/\ell \wequi \{1\}$ (use Proposition \ref{prop:tau-p-closed} and Remark \ref{rmk:spherical-bott-base-change}).
\item $m < \infty$, $\ell=2$, $n \le \infty$, $S$ defined over $\R$ (use \S\ref{subsec:reals} and Remark \ref{rmk:spherical-bott-base-change}).
\end{enumerate}

\section{Slice convergence} \label{sect:s-conv}
In this section we provide an extension of Levine's results on the convergence of the slice spectral sequence  \cite{levine2013convergence} or, more precisely, convergence of the \emph{slice tower} for motivic spectra satisfying certain $\ell$-torsion and $\rho$-torsion hypotheses; see Theorem~\ref{thm:convergence} for a precise statement. For this we use a very deep result: Voevodsky's resolution of the Milnor and Bloch-Kato conjectures \cite{ovv,voevodsky-BK}. We first set out our conventions on towers.

\begin{definition} \label{def:sep-conv}
Let $\scr C$ be a category and $E \in \scr C$.
\begin{enumerate}
\item By a \emph{tower over $E$} we mean an object $E_\bullet \in \Fun(\Z^\op, \scr C_{/E})$.
  We typically display towers as \[ \dots E_2 \to E_1 \to E_0 \to E_{-1} \to \dots \to E \quad\text{or}\quad E_\bullet \to E. \]
\item Suppose $\scr C$ is an abelian $1$-category.
  Given a tower $E_\bullet \to E$ in $\scr C$ we define the descending filtration \[ F_i E = \im(E_i \to E) \subset E. \]
  We call the tower $E_\bullet \to E$ \emph{separated} if \[ 0 = \cap_i F_i E \] and \emph{convergent} if in addition it is \emph{exhaustive}, i.e., \[ E = \cup_i F_i E. \]
\end{enumerate}
\end{definition}

Clearly if $F: \scr C \to \scr D$ is any functor and $E_\bullet \to E$ is a tower in $\scr C$, then $FE_\bullet \to FE$ is a tower in $\scr D$.


We shall utilize this definition of convergence to detect when maps are null.

\begin{lemma} \label{lemm:convergence-application}
Let $E_\bullet \to E$ be a tower in the category $\SH$ of spectra.
Denote by $E_i/E_{i+1}$ the cofiber of the canonical morphism $E_{i+1} \to E_i$.
Let $k \in \Z$ and assume that (a) the tower $\pi_k(E_\bullet \to E)$ is convergent, and (b) $\pi_k(E_i/E_{i+1}) = 0$ for every $i$.
Then $\pi_k(E) = 0$.
\end{lemma}
\begin{proof}
It suffices to prove the result for $k =0$. Let $f: \1 \to E$ be any map.
We need to show that $f = 0$.
We shall show that $f \in F_n \pi_0 E$ for all $n$; then we are done by separatedness.
By definition of exhaustiveness, we have $f \in F_N \pi_0 E$ for some $N$; hence it suffices to show that $f \in F_n \pi_0 E$ implies $f \in F_{n+1} \pi_0 E$.
Hence suppose $f \in F_n \pi_0 E$, and pick $f_n: \1 \to E_n$ such that the composite $\1 \to E_n \to E$ is homotopic to $f$; this is possible by definition of $F_n \pi_0 E$.
Since $\pi_0(E_n/E_{n+1}) = 0$ by assumption, the composite $\1 \to E_n \to E_n/E_{n+1}$ is homotopic to zero, and hence $f_n$ lifts to a map $f_{n+1}: \1 \to E_{n+1}$.
It follows that $f \in F_{n+1} \pi_0 E$.
\end{proof}

In order to apply this result, we need a good supply of convergent towers.
We shall produce them from Voevodsky's slice tower, using a strengthening of Levine's convergence theorem for the slice filtration \cite[Theorem 7.3]{levine2013convergence} that we will establish next.

Recall that there is a functorial tower \[ \SH(S) \rightarrow \SH(S)^{\Z \cup \{ \infty\} }; E \mapsto (f_\bullet E \to E) \] called the \emph{slice tower} \cite{voe-open}; see also \cite[Section 3]{1-line} for a more extensive discussion and references therein.
Recall the definition of virtual cohomological dimension from \S\ref{subsec:virtual-coh-dim}.
In order to state our result, we shall make use of the following assumptions on $E \in \SH(k)$ and $t \in \Z$:
\begin{enumerate}[(a)]
\item $E \in \SH(k)_{\ge c}$ for some $c \in \Z$.
\item For $i, j \in \Z$, $K/k$ any finitely generated, separable field extension, and $a \in \ul{\pi}_i(E)_j(K)$ we have $t^r a = 0$ for $r$ sufficiently large.
\item There exists an integer $R \gg 0$ such that the endomorphism $\rho^{R}\colon E \rightarrow E \wedge \Gmp{R}$ is homotopic to zero.
\end{enumerate}
\begin{theorem} \label{thm:convergence}
Let $k$ be a field\footnote{Not necessarily perfect, contrary to Levine's assumption in \cite{levine2013convergence}.} of exponential characteristic $e$ and $t > 0$ coprime to $e$ such that $\vcd_t(k) < \infty$.

There exists a function 
\begin{equation} \label{eq:fn}
\Z^6 \to \N, (c, d, R, i, j, M) \mapsto N(c, d, R, i, j, M)
\end{equation}
such that for every $x \in X \in \Sm_k$ with $\dim X \le d$, $(i,j,M) \in \Z^3$ and $E \in \SH(k)$ satisfying (a) and (b), the following hold:
\begin{enumerate}
\item For $n > N(c,d,R,i,j,0)$ we have\footnote{Here for a sheaf $F$ on $\Sm_k$ we denote by $F_x$ its stalk at $x \in X$.}  \[ \ul{\pi}_{i,j}(f_n(E)/\rho^R)_x = 0. \]
  In particular, the tower
\begin{equation} \label{eq:pi-tower1}
 \ul\pi_{i,j}(f_\bullet(E)/\rho^R)_x \rightarrow \ul\pi_{i,j}(E/\rho^R)_x
 \end{equation} is separated.
\item In addition, if $E$ also satisfies (c), then the morphism \[ \ul{\pi}_{i,j}(f_{M+N(c,d,R,i,j,M)} E)_x \to \ul{\pi}_{i,j}(f_{M} E)_x \] is zero.
  In particular, the tower
\begin{equation} \label{eq:pi-tower2}
 \ul\pi_{i,j}(f_\bullet E)_x \rightarrow \ul\pi_{i,j}(E)_x
 \end{equation} is separated.
\end{enumerate}
\end{theorem}

A few remarks are in order.

\begin{remark} \label{rmk:levine-conjecture}
This result is closely related to \cite[Conjecture 5]{levine2013convergence}.
Indeed the main idea in our argument is that for fields with $\vcd_2(k) < \infty$, $\rho$ is the only obstruction to nilpotence of the ideal $I=I(k) \subset GW(k)$.
In particular, under assumption (c), $I$ acts nilpotently on each $\ul{\pi}_{i,j}(E)_x$, and thus in particular this module is $I$-adically complete.
Levine's conjecture thus predicts our separatedness result.

Conversely, over fields of finite $\vcd_2$, the $I$-adic and $\rho$-adic filtrations are commensurate on $K_*^{MW}$; it thus seems justified to think (over such fields) of derived $\rho$-completion as a form of $I$-adic completion.
We thus view our results as establishing a derived version of Levine's conjecture.
\end{remark}

\begin{remark} \label{rm:fn} The function~\eqref{eq:fn} indicates the dependence of the number $N(c, d, R, i, j, M)$ on the connectivity of $E$ (given by $c$), the bidegrees we are interested in (given by $(i, j)$), the effective cover of $E$ we are taking (given by $M$) and the dimension of the scheme (or, rather, the point; this is given by $d$). Note that this function does not depend on the number $r$ that appears in condition (b).
\end{remark}

\begin{remark} \label{rmk:exhaustive}
The slice tower is always exhaustive (see, for example, \cite[Lemma 3.1]{rso-solves}); hence the theorem implies that the towers \eqref{eq:pi-tower1}, \eqref{eq:pi-tower2} are convergent.
\end{remark}

\begin{remark} \label{rmk:invertible}
If $d$ is coprime to $t$, then (b) implies multiplication by $d$ is an isomorphism on $\ul{\pi}_{i,j}(E)$ for all $i,j$, and hence so is $d\colon E \rightarrow E$.
In particular, under the hypotheses of Theorem~\ref{thm:convergence}, we have that $E \in \SH(k)[1/e]$.
\end{remark}

\begin{remark} \label{rmk:convergence-simp}
It follows from Remark \ref{rmk:rho-completion} that condition (c) is vacuous if $k$ is unorderable (e.g., $\cd_2(k) < \infty$), or if $t$ is odd and $E \in \SH(k)[1/2]^+$.
In these cases the statement of Theorem \ref{thm:convergence} is not quite optimal; in fact the proof shows that \[ \ul{\pi}_{i,j}(f_{M} E)_x = 0 \] for $M \gg 0$ (depending on $i,j,c,d$).
\end{remark}

\begin{remark} \label{rmk:nilpotence}
If $a \in \pi_{p,q}\1$ and $F \in \SH(k)$, then standard arguments show that $a^2: F/a \to \Sigma^{2p,2q} F/a$ is the zero map (see e.g., \cite[Lemma 5.2]{rondigs2019remarks}).
It follows that for $E \in \SH(k)_{\ge c}$ Theorem \ref{thm:convergence} applies to $E/(\rho^a,t^b)$ (and also $E/t^b$ or $(E/t^b)^+$, if Remark \ref{rmk:convergence-simp} applies).
\end{remark}

Our result implies convergence of the slice spectral sequence in novel cases.
We record the following, even though we do not use it in the rest of the article.
\begin{corollary}
Let $k$ be a field of exponential characteristic $e$ and $t > 0$ coprime to $e$ such that $\vcd_t(k) < \infty$.
Suppose $E \in \SH(k)_{\ge c}$ for some $c \in \Z$.
\begin{enumerate}
\item The map $E_{t,\rho}^\comp \xrightarrow{\scomp} \scomp(E)_{t,\rho}^\comp$ induces an isomorphism on $\pi_{**}$;\NB{It is also the case that $E_{t,\rho}^\comp \wequi \lim_{m,n}\scomp(E/(t^n,\rho^m))$, but this seems like a less useful statement.}
here $\scomp$ denotes the slice completion functor \cite[\S 3]{1-line}.
\item There is a conditionally convergent spectral sequence \[ \pi_{p,n}(s_q(E)_{t,\rho}^\comp) \Rightarrow \pi_{p,n}(E_{t,\rho}^\comp), \] using the indexing conventions of \cite{1-line}.
\end{enumerate}
\end{corollary}
\begin{proof}
(1) It is enough to show that $E/(\rho^n,t^m) \to \scomp(E)/(\rho^n,t^m)$ is a $\pi_{**}$-isomorphism.
This is immediate from Theorem \ref{thm:convergence}(1) and Remark \ref{rmk:exhaustive} (and Remark \ref{rmk:nilpotence}, which tells us that Theorem \ref{thm:convergence} applies).

(2) We consider the completed slice tower \[ \dots \to f_n(E)_{t, \rho}^\comp \to f_{n-1}(E)_{t,\rho}^\comp \to \dots \to E_{t,\rho}^\comp. \]
The cones of the maps in this tower are given by the $(t,\rho)$-completed slices $s_n(E)_{t, \rho}^\comp$.
To prove conditional convergence for the corresponding spectral sequence displayed in (2), we need to show that \[ \pi_{**}(\lim_n(f_n(E)_{t,\rho}^\comp)) \wequi 0 \quad\text{and}\quad \pi_{**}(\colim_n (f_n(E)_{t,\rho}^\comp)) \wequi \pi_{**}(E_{t,\rho}^\comp). \]
Since limits commute we have the fiber sequence \[ \lim_n(f_n(E)_{t,\rho}^\comp) \wequi (\lim_n f_n(E))_{t,\rho}^\comp \to E_{t,\rho}^\comp \xrightarrow{\scomp} \scomp(E)_{t,\rho}^\comp. \]
Hence the claim about limits reduces to (1).
For the claim about colimits, observe that for $X \in \Sm_k$ and $n<w$ we have
\begin{gather*}
  \map(X(w), f_n(E)_{t,\rho}^\comp) \\
  \wequi \lim_m \cof(\map(X(w), \Sigma^{-m,-m} f_n(E)/t^m) \xrightarrow{\rho^m} \map(X(w), f_n(E)/t^m)) \\
  \wequi \lim_m \cof(\map(\Sigma^{m,m} X(w), f_n(E/t^m)) \xrightarrow{\rho^m} \map(X(w), f_n(E/t^m))) \\
  \wequi \lim_m \cof(\map(\Sigma^{m,m} X(w), E/t^m) \xrightarrow{\rho^m} \map(X(w), E/t^m)) \\
  \wequi \map(X(w), E_{t,\rho}^\comp),
\end{gather*}
where $X(w) := \Sigma^{0,w}\Sigma^\infty_+ X$.
Since the $X(w)$ are compact generators of $\SH(k)$, we deduce that in fact \[ \colim_n (f_n(E)_{t,\rho}^\comp) \wequi E_{t,\rho}^\comp. \]
\end{proof}
\begin{remark}
Under the assumptions of the corollary, it is in fact the case that $E_{t,\rho}^\comp \to \scomp(E)_{t,\rho}^\comp$ is an equivalence.
This requires a slightly more elaborate argument and will be treated elsewhere.
\end{remark}

In the rest of this section we prove Theorem \ref{thm:convergence}, adapting the argument of Levine \cite[Theorem 7.3]{levine2013convergence}.
Without loss of generality, we make the following standing assumptions:
\begin{enumerate}
\item $c = 0$,
\item $k$ is \emph{perfect} (using \cite[Corollary 2.1.7]{elmanto2018perfection} and Remark \ref{rmk:invertible}),
\item $k$ is infinite (using standard transfer arguments; see \cite[Appendix A]{levine2013convergence} for details),
\item $x$ is a generic point, so $X_x$ is the spectrum of a field of transcendence degree $\le d$ over $k$ (using unramifiedness of homotopy sheaves \cite[Lemma 6.4.4]{morel-conn}).
\end{enumerate}

With these assumptions at play, we quickly review Levine's approach to studying the slice filtration via the \emph{simplicial filtration} as in \cite[Section 4]{levine2013convergence}. For $E \in \SH(k)$ and for any $M \ge 0, X \in \Sm_k$, consider the mapping spectrum \[ f_ME(X)= \map(\Sigma^{\infty}X_+, f_ME) \in \SH; \] the functor $X \mapsto f_ME(X)$ is an $\A^1$-invariant Nisnevich sheaf of spectra. As elaborated in \cite[Section 3]{levine2013convergence}, we have an augmented simplicial spectrum of the form
\[
E^{(M)}(X, \bullet) \rightarrow f_ME(X),
\]
which is a colimit diagram. 
Under the Dold-Kan correspondence \cite[Theorem 1.2.4.1]{HA} we get the associated filtered spectrum
\begin{equation} \label{eq:simpfilt}
\sk_0 E^{(M)}(X, \bullet) \rightarrow \sk_1 E^{(M)}(X, \bullet) \rightarrow \cdots \to \sk_k E^{(M)}(X, \bullet) \rightarrow \cdots \to f_ME(X).
\end{equation}

Noting that \[ \ul{\pi}_i(E)_{-j}(K) = [\Sigma^i \Gmp{j} \wedge K_+, E] \] (where $K/k$ is a field extension) we obtain for $j \ge 0$ a spectral sequence of the form
\begin{equation} \label{eq:simpss}
E^1_{p,q}(K, E, M, j) \Rightarrow \ul{\pi}_{p+q}(f_ME)_{-j}(K).
\end{equation}
As usual, defining 
\[
F^{\simp}_{k}\ul{\pi}_{i}(f_ME)_{-j}(K):= \mathrm{Im}(\pi_{i}\sk_k(E^{(M)}(\Gmp{j} \wedge K_+, \bullet)) \rightarrow \ul{\pi}_{i}(f_ME)_{-j}(K)),
\]
we have the exhaustive increasing filtration
\begin{equation} \label{eq:pi-simpss}
F^{\simp}_{0}\ul{\pi}_{i}(f_ME)_{-j}(K) \rightarrow \cdots F^{\simp}_{k}\ul{\pi}_{i}(f_ME)_{-j}(K) \rightarrow F^{\simp}_{k+1}\ul{\pi}_{i}(f_ME)_{-j}(K) \cdots \rightarrow \ul{\pi}_{i}(f_ME)_{-j}(K),
\end{equation}
and the associated graded identifies with the $E^{\infty}$ page \cite[sentence before Lemma 4.4]{levine2013convergence}:
\[
gr_p^{\simp}\ul{\pi}_{i}(f_ME)_{-j}(K) \cong E^{\infty}_{p,i-p}(K,E,M,j).
\]

The next lemma concerns the spectral sequence~\eqref{eq:simpss} and the corresponding filtration on sections of homotopy sheaves. To state the result we need some notation. Suppose that $F, G, H$ are strictly $\A^1$-invariant sheaves on $\Sm_k$, $H$ has transfers in the sense of \cite[Chapter 4]{morel-book}, and there is a pairing $F \otimes G \rightarrow H$. Then for a finitely generated separable extension $K'$ of $k$, recall that we have the subgroup
\[
[ F G]^{tr}(K) := \langle tr_{K'/K} F(K')G(K') \rangle_{K'/K\,\text{finite}} \subset H(K).
\]
We refer to \cite[Section 4]{bachmann-hurewicz} or \cite[Section 7]{levine-gw} for details on this construction. The case that we are interested in will be $F = \ul{K}^{MW}_n \cong \ul{\pi}_0(\1)_n$ acting on $G=\ul{\pi}_{i}(E)_j$.

\begin{lemma}[Levine] \label{lemm:levine-simp} Suppose that $E \in \SH(k)_{\geq 0}$, i.e., (a) holds. Then for all $M \geq 0, j \ge 0$ and all $i \in \Z$ we have:
\begin{enumerate}
\item The spectral sequence~\eqref{eq:simpss} is convergent.
\item There is an inclusion of abelian groups 
\[
E^1_{p,q}(K, E, M, j) \subset \bigoplus_{w \in (\Delta^p_K, \partial \Delta_K^p)^{(M)}} \ul{\pi}_{q+M}(E)(K(w))_{-M-j}
\]
\item \[ F^{\simp}_{*}\ul{\pi}_{i}(f_ME)_{-j}(K) = 0 \text{ if } * < M, \] and \[ F^{\simp}_{*}\ul{\pi}_{i}(f_ME)_{-j}(K) = \ul{\pi}_{i}(f_ME)_{-j}(K) \text{ if } * > M+i. \]
\item Under the canonical map $\ul{\pi}_i(f_ME)_{-j} \rightarrow \ul{\pi}_i(E)_{-j}$, the image of $F^{\simp}_{M}\ul{\pi}_{i}(f_ME)_{-j}(K)$ in $\ul{\pi}_i(E)_{-j}(K)$ is the subgroup 
\[
[K_M^{MW} \ul{\pi}_i(E)_{-M-j}]^{tr}(K).
\]
\end{enumerate}
\end{lemma}  

\begin{proof} The first two points are covered in \cite[Lemma 4.4]{levine2013convergence}, the third is immediate from (2) (see also \cite[Lemma 5.1]{levine2013convergence}) and the last point is \cite[Theorem 5.3]{levine2013convergence}.
\end{proof}

\begin{lemma} \label{lemm:convergence-permanence}
Let $E \in \SH(k)$ and $M \in \Z$. In the notation of Theorem~\ref{thm:convergence}:
\begin{enumerate}
\item[(a)] If $E$ satisfies (a) then so does $f_M E$.
\item[(b)] If $E$ satisfies (b) then so does $f_M E$.
\end{enumerate}
\end{lemma}
\begin{proof}
(a) We may assume (replacing $E$ by $E \wedge \Gmp{-M}$) that $M=0$.
The claim is then immediate from \cite[Proposition 4(3)]{bachmann-very-effective}.

(b) We may assume (replacing $E$ by $E \wedge \Sigma^{-i}\Gmp{j}$) that $i=j=0$.
If $M < 0$ there is nothing to prove, since $\ul{\pi}_0(f_ME)_0 \wequi \ul{\pi}_0(E)_0$.
Thus we may assume that $M \ge 0$. Consider the filtration~\eqref{eq:pi-simpss} of $\ul{\pi}_0(f_M E)_0(K)$.
Since this filtration is finite and exhaustive, it suffices to prove the claim on the level of associated graded groups, which are the $E^{\infty}$-terms of the spectral sequence~\eqref{eq:simpss}. The claim follows since the $E^{\infty}$-terms are subquotients of the $E^1$-terms, which by Lemma~\ref{lemm:levine-simp} are in turn subgroups of sums of groups of the form $\ul{\pi}_{a}(E)(K(w))_{b}$, which satisfy (b) by assumption.

%
\end{proof}

\begin{remark} \label{rmk:fM-rho}
It is not clear that if $E$ satisfies condition (c), then so does $f_M E$.
Instead the composite \[ f_ME \xrightarrow{\rho^R} f_M(E) \wedge \Gmp{R} \to f_{M-R}(E) \wedge \Gmp{R} \] is zero.
We are not going to use this observation (explicitly).
\end{remark}

\begin{lemma} \label{lemm:convergence-1}
Let $E$ satisfy conditions (b) and (c) of Theorem \ref{thm:convergence}.
There exists a function $M_0(d, R)$ (independent of $E$) such that for $n \ge M_0(d, R)$, $\trdeg(K/k) \leq d$, and $i, j \in \Z$ we have \[ [\ul{K}_n^{MW} \ul{\pi}_i(E)_j]^{tr}(K) = 0. \]
\end{lemma}
\begin{proof}
We define the function by
\begin{equation} \label{eq:modr}
M_0 := M_0(d, R) := \vcd_t(k) + d + R + 1.
\end{equation}
Let $K/k$ be of transcendence degree $\le d$ and $K'/K$ finite.
Assumption (b) implies that every element of $\ul{\pi}_i(E)_j(K')$ is a sum of elements which are $\ell^m$-torsion, for various $\ell|t$ and $m \ge 0$.
Put $F_* = K_*^{MW}(K')/(\rho^R, \ell^m)$.
Together with assumption (c), the above observation implies that it suffices to establish the vanishing $F_{M_0} = 0$.

If $\ell$ is odd then $F_*$ is a quotient of $K_*^{MW}(K')[1/2]$ and so splits into $+$ and $-$ parts; moreover $\rho$ is an isomorphism on the $-$ part (see Remark \ref{rmk:rho-completion}(2)) and thus $F_*^-$ is zero. Hence, it is enough to show that $K_{M_0}^M(K')/\ell^m = 0$.
This is true by choice of $M_0$, the fact that $\mathrm{cd}_\ell(K') \leq d + \mathrm{cd}_\ell(k) = d + \vcd_\ell(k)$ \cite[Tag 0F0T]{stacks}, and the comparison between Milnor $K$-theory and \'etale cohomology \cite{voevodsky-BK}.

Now let $\ell=2$.
Note that if $M$ is an $A$-module and $a \in A$ satisfies $a^N M = 0$ for some $N$, then $M=0$ if and only if $M/a = 0$.
Hence we may assume that $m=1$. Note that $h$ is nilpotent on $F_*$: \[ 0 = (\eta \rho)^R = (\lra{-1} - 1)^R = (h-2)^R = h^R. \] It thus suffices to show that $(F_{M_0})/h = 0$.
But $K_*^{MW}/h = I^*$ \cite[Th\'eor\'eme 2.4.]{morel-ideal}, so it is enough to show the equality \[ \rho^R I(K')^{\vcd_2(K')} = I(K')^{\vcd_2(K')+R}. \] Under the isomorphism of \emph{loc. cit.}, the element $\rho$ is sent to the element $2 \in I(k) \subset W(k)$. Hence we need to check that $I(K')^{\vcd_2(K')+R} = 2^R I(K')^{\vcd_2(K')}$.
For a proof see p.\ 619 in \cite{elman-lum-2cohom}.
\end{proof}

The following result is the key step in the proof of Theorem \ref{thm:convergence}.
It establishes a vanishing region in the homotopy sheaves of sufficiently effective (and connective) spectra.
\begin{lemma} \label{lemm:convergence-2}
There exists a function $M(d, R, r, s)$ such that for \[ E \in \SH(k)^\eff(M(d,R,r,s)) \] satisfying (a)\footnote{Recall our standing convention that $c=0$, so this means $E \in \SH(k)_{\ge 0}$.}, (b) and (c) of Theorem \ref{thm:convergence}, and $i \le r, 0 \le j \le s$ we have \[ \ul{\pi}_i(E)_{-j}(K) = 0, \]
whenever $\trdeg(K/k) \leq d$.
\end{lemma}
\begin{proof}
We shall define $M(d,R, r, s)$ by induction on $r$.
If $i < 0$ then $\ul{\pi}_i(E)_{-j}(K)=0$ for any $j, K$ by assumption, and hence $M(d,R, r, s)=0$ works for all $r < 0$, $s \in \Z$.

Now suppose that $M(d,R, r',s)$ has been defined for all $r' < r$.
We first want to define $M(d,R,r,0)$; so we need to investigate $\ul{\pi}_r(E)_0(K)$. Consider the filtration~\eqref{eq:pi-simpss} on $\pi_r(f_n E)_0(K)$ with associated graded $gr_p^{\simp}\ul{\pi}_{r}(f_nE)_{0}(K)$.
It follows from Lemma \ref{lemm:levine-simp}(2) that these groups vanish when:
\begin{enumerate}
\item $p < n$ by codimension reasons, or
\item $p > n+r$ by connectivity of $E$.
\end{enumerate}
Conversely, whenever there is a non-zero contribution we must have $(*)$ $n \le p \le n+r$, and the contribution comes from some $\ul{\pi}_{r'}(E)_{-n}(K(w))$, where \[ \trdeg(K(w)/K) = p-n \le r \text{ and } r' = n+r-p \le r. \]
Here the last two inequalities just come from $(*)$.

We define \[ M(d,R,r,0) = \max\{M(d+r,R, r-1, M_0(d,R)), M_0(d,R)\}, \]
where $M_0(d,R)$ is the function from~\eqref{eq:modr}.
Set $n=M_0(d,R)$.
By assumption, $E$ is at least $n$-effective, so $f_nE \wequi E$ and the map \[ \alpha_n: \pi_r(f_nE)_0(K) \to \pi_r(E)_0(K) \] is an isomorphism.
We claim that for $p>n$, we have $gr_p^{\simp}\ul{\pi}_{r}(f_nE)_{0}(X)= 0$.
Indeed by the above discussion any contribution to this group arises from $\ul{\pi}_{r'}(E)_{-n}(L)$, where $r'<r$ and $\trdeg(L/k) \le d+r$.
This vanishes by induction and the construction of $M$.
Now, part (3) of Lemma~\ref{lemm:levine-simp} implies that the image of $\alpha_n$ is given by $F_n^{\simp}\ul{\pi}_{i}(f_nE)_{0}(K)$, which is identified in part (4) as the group \[ [K_n^{MW} \ul{\pi}_r(E)_{-n}]^{tr}(K), \] which vanishes by Lemma \ref{lemm:convergence-1}.
Hence our definition of $M(d,R,r,0)$ has the desired property.

Finally, since $\ul{\pi}_i(E)_{-j} = \ul{\pi}_i(E \wedge \Gmp{-j})_0$, we can define $M(d,R,r,s) = M(d,R,r,0) + s$. This concludes the construction of the desired function.
\end{proof}

\begin{proof}[Proof of Theorem \ref{thm:convergence}.]
Let us call a tower of sheaves $\ul{F}_\bullet$ \emph{locally null} if there exists a function $N(d)$ such that for $\trdeg(K/k) \le d$ and $i \ge N(d)$ we have $\ul{F}_{i}(K) = 0$.

(1) Under our standing assumptions (in particular (4)), the claim is equivalent to showing that $\ul{\pi}_i((f_\bullet E)/\rho^R)_j$ is locally null (with $N$ independent of $E$).
Replacing $E$ by $E \wedge \Gmp{j}$ (which still satisfies (a), (b)) replaces $\ul{\pi}_i(f_M(E)/\rho^R)_0$ by $\ul{\pi}_i(f_{M-j}(E)/\rho^R)_j$.
It thus suffices to treat the case $j=0$.
Note that $f_M(E)/\rho^R$ is $(M-R)$-effective, satisfies conditions (a) and (b) by Lemma \ref{lemm:convergence-permanence}, and satisfies condition (c) by Remark \ref{rmk:nilpotence} (with $2R$ in place of $R$).
Lemma \ref{lemm:convergence-2} thus shows that $\ul{\pi}_i(f_M(E)/\rho^R)_0(K) = 0$ as soon as $M \ge M(d,2R, i, 0) + R$.
The claim follows.

(2)
Consider the following commutative diagram
\begin{equation*}
\begin{CD}
f_M(E) @>>> f_M(E)/\rho^R @>a>> f_{M-R}(E) \\
@VVV         @VVV                @VVV      \\
E      @>>> E/\rho^R      @>b>> E.
\end{CD}
\end{equation*}
Here the vertical maps are the canonical ones, and the horizontal maps in the left hand square are the projections.
The map $b$ is a splitting $E/\rho^R \wequi E \vee \Sigma E \wedge \Gmp{-R} \to E$ (using condition (c)), and the map $a$ is the unique one making the right hand square commute 
(using that $f_M(E)/\rho^R$ is $(M-R)$-effective).
The bottom horizontal composite is $\id_E$ by construction, and hence the top horizontal composite is the canonical map $f_M(E) \to f_{M-R}(E)$.
Since $\ul{\pi}_i(f_\bullet(E)/\rho^R)_j$ is locally null by (1), we deduce that for $M$ sufficiently large depending only on $i, j, d$, the map \[ \ul{\pi}_i(f_{M+R} E)_j(K) \to \ul{\pi}_i(f_{M} E)_j(K) \]  factors through \[ \ul{\pi}_i(f_{M+R}(E)/\rho^R)_j(K) = 0. \]
This implies the claim.

Separatedeness of both towers follows immediately.
This concludes the proof.
\end{proof}

\section{Spheres over fields}
In this section we treat a special case of our main result to which we will reduce the general case.
Throughout $S=Spec(k)$, where $k$ is a field of exponential characteristic $e$, $\ell \ne e$ is a prime and 
\[ \tau: \1/(\ell^{\nu}, \rho^{\mu}) \to \1/(\ell^{\nu}, \rho^{\mu})(r) \] 
is a $\tau$-self map (see Definition \ref{def:tau-self-map}).
Here $0 \le \mu,\nu \le \infty$.

The following is the key technical result.
\begin{lemma} \label{lemm:key}
Suppose that $\vcd_\ell(k) < \infty$.
Then \[ \lim_n \left[ f_n(\1)/(\ell^\nu,\rho^\mu)[\tau^{-1}]/(\ell,\rho) \right] = 0. \]
\end{lemma}
\begin{proof}
Write $L = \lim_n L_n$ for the limit in question.
For $X \in \Sm_k$, $w \in \Z$, spectra of the form $X(w) := \Sigma^{0, w} \Sigma^\infty X_+$ generate $\SH(k)$ (see, for example, \cite[Theorem 9.2]{cell} or \cite[Proposition 6.4]{hoyois-sixops} when $G$ is the trivial group), 
and hence it suffices to show that \[ [X(w)[i], L] = 0 \] for all $w, i$.
By the Milnor exact sequence \cite[Proposition VI.2.15]{goerss2009simplicial} \[ 0 \to \mathrm{lim}^1_n[X(w)[i+1], L_n]  \rightarrow [X(w)[i], L] \rightarrow \mathrm{lim}_n [X(w)[i], L_n] \to 0, \] it is enough to show that for $(i, w)$ fixed and $n$ sufficiently large we have $[X(w)[i], L_n] = 0$.

Consider the descent spectral sequence
\[
H^p_{\mathrm{Nis}}(X, \ul{\pi}_{q,w}(L_n)) \Rightarrow [X(w)[q-p] , L_n],
\]
which is strongly convergent due to the finite cohomological dimension of the Nisnevich site (see e.g., the proof of \cite[Proposition 4.3]{elso}).
This implies that it suffices to prove the following claim:
\begin{itemize}
\item For any $N \in \Z$ and $n$ sufficiently large (depending on $N$), all $w \in \Z$ and $k \le N$ we have $\ul{\pi}_{k,w}(L_n)|_{X_\Nis} = 0$.
\end{itemize}
\tombubble{Or use that Zariski and Nisnevich cohomology coincide.}
By unramifiedness of homotopy sheaves \cite[Lemma 6.4.4]{morel-conn}, it is enough to show that $\ul{\pi}_{k,w}(L_n)_\xi = 0$ for generic points $\xi$ of étale extensions of $X$.
Since homotopy sheaves commute with colimits, and $w \in \Z$ is arbitrary, for this it suffices to show that \[ \ul{\pi}_{k,w}(f_n(\1)/(\ell^\nu,\rho^\mu,\ell,\rho))_\xi = 0. \]
Up to adding a constant to $n$ (which only depends on $\mu$), it is thus enough to show the vanishing \[ \ul{\pi}_{k,w}(f_n(\1)/(\ell,\rho))_\xi = 0. \]

Using Theorem \ref{thm:convergence}(1), we find that the tower of abelian groups
\[
\ul{\pi}_{k,w}(f_{\bullet}(\1)/(\ell,\rho))_{\xi} \rightarrow \ul{\pi}_{k,w}(f_n(\1)/(\ell,\rho))_\xi
\]
is convergent (here $\bullet \ge n$, and the tower is trivially exhaustive).
Hence by Lemma \ref{lemm:convergence-application}, it suffices to prove the following claim:
\begin{itemize} 
\item For any $N \in \Z$ there exists $n$ (depending on $N$ and $\trdeg(\xi/k)$) such that for all $m \ge n$, $k \le N$ and $w \in \Z$ we have
\[ \ul{\pi}_{k,w}(s_m(\1)/(\ell,\rho))_\xi = 0.  \] 
\end{itemize}
By \cite[Theorem 2.12]{1-line}, each slice $s_m(\1)[1/e]$ is a finite sum of suspensions of motivic cohomology spectra \[ \Sigma^{m+s,m} H\Z/l, \] for certain $s \ge 0, l \ge 0$.

As before, up to possibly adding a constant to $n$, it is thus enough to show that \[\ul{\pi}_{k,w}(\Sigma^{m+s,m} H\Z/(\ell,\rho))_\xi = 0\] for $m$ sufficiently large.
Now \[ \ul{\pi}_{k,w}(\Sigma^{m+s,m} H\Z/\ell)_\xi = H^{m+s-k}(\xi, \Z/\ell(m-w)) \] and so we need to show that multiplication by $\rho$ induces an isomorphism on these groups, 
for $m$ sufficiently large (depending on $N$ and $\trdeg(\xi/k)$), all $w \in \Z$ and all $k \le N$, $s \ge 0$.
In particular $s-k \ge -N$.
Note also that we have $\vcd_\ell(\xi) \le \vcd_\ell(k) + \trdeg(\xi/k)$ by \cite[Tag 0F0T]{stacks}.
The required vanishing thus follows from Lemma \ref{lemm:vcd-iso}: we may put $m \ge n := \vcd_\ell(k) + \trdeg(\xi/k) + N + 1$. This concludes the proof.
\end{proof}

\begin{corollary} \label{cor:sphere-inversion}
If $\vcd_\ell(k) < \infty$, then $\1/(\ell^\nu,\rho^\mu)[\tau^{-1}]/(\ell,\rho)$ is étale local.
\end{corollary}
\begin{proof}
We have a cofiber sequence of towers
\begin{equation*}
\begin{CD}
\dots @>>> f_2(\1) @>>> f_1(\1) @>>> f_0(\1)=\1 \\
@.          @VVV        @VVV         @VVV       \\
\dots @>>> \1      @>{\id}>> \1 @>{\id}>> \1    \\
@.          @VVV        @VVV         @VVV       \\
\dots @>>> f^2(\1) @>>> f^1(\1) @>>> f^0(\1),
\end{CD}
\end{equation*}
where $f^i = \id/f_i$.
Smashing with $\1/(\ell^\nu,\rho^\mu)[\tau^{-1}]/(\ell,\rho)$ and applying the (exact) inverse limit functor, we obtain a cofiber sequence \[ \lim_n \left[ f_n(\1)/(\ell^\nu,\rho^\mu)[\tau^{-1}]/(\ell,\rho) \right] \to \1/(\ell^\nu,\rho^\mu)[\tau^{-1}]/(\ell,\rho) \xrightarrow{\alpha} \lim_n \left[ f^n(\1)/(\ell^\nu,\rho^\mu)[\tau^{-1}]/(\ell,\rho) \right]. \]
By Lemma \ref{lemm:key}, the map $\alpha$ is an equivalence.
Since étale local spectra are closed under limits and extensions and each $f^n(\1)$ is a finite extension of the slices $s_m(\1)$ (recall that $\1$ is effective), it suffices to show that $s_i(\1)/(\ell^\nu,\rho^\mu)[\tau^{-1}]/(\ell,\rho)$ is étale local for every $i$.
By the form of the slices of $\1$ recalled in the proof of Lemma \ref{lemm:key}, it suffices to show that \[ H\Z/(\ell^\nu,\rho^\mu)[\tau^{-1}]/(\ell,\rho) \] is étale local.
This holds by the definition of a $\tau$-self map.
\end{proof}

\section{Main result}
In this section we establish our main étale localization results.
Before doing so, we need some preliminaries.

\begin{proposition} \label{prop:compact-gen}
Let $1/\ell \in k$, $\vcd_\ell(k) < \infty$ and $S$ a finite type $k$-scheme.
\NB{Not minimal assumptions...}
If $X \in \Sm_S$ is quasi-separated and $w \in \Z$, then $\Sigma^\infty_+ X/(\ell, \rho) \wedge \Gmp{w} \in \SH_\et(S)_{\ell,\rho}^\comp$ is compact.
\end{proposition}
\begin{proof}
Since the pullback $\SH_\et(S)_{\ell,\rho}^\comp \to \SH_\et(X)_{\ell,\rho}^\comp$ preserves colimits, we may assume that $S=X$.

First suppose that $\cd_\ell(k) < \infty$ (e.g., $\ell$ odd).
It is enough to show that $\Sigma^\infty_+ S/\ell$ is compact in $\SH_\et(S)_\ell^\comp$.
This is proved in \cite[Corollary 5.7 and Example 5.9]{bachmann-SHet}.

Now let $\ell=2$. 
\tombubble{this argument is a bit awkward, I wonder if there is a better formulation...}
Write $\SH(S_\et)$ for the category of hypersheaves of spectra on the small étale site of $S$ (see e.g., \cite[\S2.2]{bachmann-SHet} for definitions).
We have the functors \[ E \mapsto \tau_{\ge n}E, \tau_{\le n}E, \pi_n(E): \SH(S_\et) \to \SH(S_\et) \] coming from the standard $t$-structure; they all preserve filtered colimits (for $\tau_{\le n}$ this follows from \cite[Proposition 1.3.2.7(2)]{SAG} and the other cases follow from this).
We thus have restricted functors \[ E \mapsto \tau_{\ge n}(E/\ell), \tau_{\le n}(E/\ell), \pi_n(E/\ell): \SH(S_\et)_\ell^\comp \to \SH(S_\et)_\ell^\comp \] with the same property.
Moreover $E/\ell \wequi \lim_n \tau_{\le n}(E/\ell)$ \cite[Lemma 2.16]{bachmann-SHet}.
Via \cite[Theorem 6.6]{bachmann-SHet} we transplant these functors to $\SH_\et(k)_\ell^\comp$.
In order to prove that $S/(\ell,\rho)$ is compact, it is enough to show that the functor $E \mapsto [S/\rho, E/\ell]$ preserves filtered colimits.
Using the Postnikov completeness result just recalled, there is a conditionally convergent spectral sequence 
\[ [S/\rho, \pi_i(E/\ell)[j]] \Rightarrow [S/\rho, E/\ell[i+j]]. \]
By standard arguments it is enough to show that the spectral sequence converges strongly for every $E$, and that $[S/\rho, \pi_i(E/\ell)[j]]$ is compatible with filtered colimits in $E$.
Consider $F = \pi_0(E/\ell) \in \SH(S_\et)^\heartsuit$.
This is a sheaf of $\Z/\ell^2$-modules on the small étale site of $S$.
There is a canonical filtration $0 \to K \to F \to F/\ell \to 0$ with $K, F/\ell$ sheaves of $\Z/\ell$-modules and an associated long exact sequence 
\[ \dots \to [S/\rho, K] \to [S/\rho, F] \to [S/\rho, F/\ell] \to \dots. \]
It will thus be enough to show that the functor \[ F \mapsto [S/\rho, F[i]] \] is compatible with filtered colimits of sheaves of $\Z/\ell$-vector spaces on the small étale site of $S$ and vanishes for $i<0$ or $i > \cd_2(S[\sqrt{-1}])$ (here we note that $\cd_2(S[\sqrt{-1}]) < \infty$, e.g., by \cite[Example 5.9]{bachmann-SHet}).

We claim that there is a long exact sequence \[ \dots \to H^i(S, F) \xrightarrow{(-1)} H^{i+1}(S, F) \to [S/\rho, F[i]] \to H^{i+1}(S, F) \to \dots. \]
Since étale cohomology commutes with filtered colimits (see e.g., \cite[A.2.3.2(1)]{SAG}), the result will follow from Lemma \ref{lemm:vcd-vanishing} below.
To prove the claim, consider the sequence of functors and adjoints \[ \SH(S_\et)_\ell^\comp \adj \SH_\et(S)_\ell^\comp \adj \DM_\et(S, \Z/\ell). \]
There is an object $F \in \DM_\et(S, \Z/\ell)$ with image $F$ in $\SH(S_\et)_\ell^\comp$; we are thus reduced to proving the analogous result in $\DM_\et(S, \Z/\ell)$.
To conclude we note that $\rho\colon \Gmp{-1} \rightarrow \1$ induces multiplication by $(-1)$ in $\DM_\et(k, \Z/2) \wequi D(k_\et, \Z/2)$.\tombubble{reference?}.
\end{proof}

Our next result follows from \cite[Theorem 99.13]{MR2427530} in the case of fields.
\begin{lemma} 
\label{lemm:vcd-vanishing} 
Let $X$ be a scheme with $1/2 \in X$ and set $X' = X[\sqrt{-1}]$.
Suppose $F$ is an étale sheaf of $\Z/2$-vector spaces on $X$.
Then for $i > \cd_2(X')$ the map \[ (-1): H^i_\et(X, F) \to H^{i+1}_\et(X, F) \] is an isomorphism.
\end{lemma}
\begin{proof}
If $X=X'$ or $\cd_2(X') = \infty$, there is nothing to prove.
We may thus assume that $Aut(X'/X)$ is the group $C_2$ of order $2$, and that $\cd_2(X') < \infty$.
Consider the strongly convergent Hochschild--Serre spectral sequence \cite[Theorem III.2.20]{MilneEtaleCohomology} \[ H^p(C_2, H^q(X', F)) \Rightarrow H^{p+q}(X, F). \]
It is a module over the same spectral sequence with $F=\Z/2$.
The class $(-1) \in H^1(X, \Z/2)$ is detected by a class $\alpha \in H^1(C_2, H^0(X, \Z/2))$ coming from the isomorphism $H^*(C_2, \Z/2) \wequi H^*(\R\P^\infty, \Z/2) \wequi \Z/2[\alpha]$.
The class $\alpha$ has the property that if $G$ is any $\Z/2[C_2]$-module and $i > 0$, the multiplication map $\alpha: H^i(C_2, G) \to H^{i+1}(C_2, G)$ is an isomorphism.\footnote{Indeed $A := \Z/2[C_2] \wequi \Z/2[\epsilon]/\epsilon^2$, so any finitely generated $A$-module is a sum of copies of $A$ and $\Z/2$. For such modules the claim holds. A general $A$-module is a filtered colimit of finitely generated ones, and cohomology commutes with filtered colimits, so the claim follows in general.}
We deduce that $\alpha$ induces an isomorphism on all columns of the spectral sequence, except possibly the one containing $H^0(C_2, H^*(X', F))$.
The result follows.
\end{proof}

We deduce a weak form of Theorem \ref{thm:apps}(3) from the introduction, which we will use in the proof of our main result.
(The strong form will be deduced later.)
\begin{corollary} \label{cor:etale-locn-smashing}
Let $k$ be a field of exponential characteristic $e \ne \ell$ and $\vcd_\ell(k) < \infty$.
Then $L_\et: \SH(k)_{\ell,\rho}^\wedge \to \SH(k)_{\ell,\rho}^\wedge$ is a smashing localization, or in other words for $E \in \SH(k)_{\ell,\rho}^\wedge$ we have $L_\et(E) \wequi E \wedge L_\et(\1)$.
\end{corollary}
\begin{proof}
The functor $L_\et: \SH(k)_{\ell,\rho}^\wedge \to \SH_\et(k)_{\ell,\rho}^\wedge$ will identify the target as (highly structured) modules over the étale-local sphere 
if we can show that both the source and target are compact-rigidly generated, 
see e.g., \cite[Lemma 22]{bachmann-hurewicz} .
This is indeed the case; see \cite[Example 2.3]{bachmann-SHet} (to obtain a family of generators), \cite[Corollary B.2]{levine2013algebraic} (to see that the generators are rigid) 
and Proposition \ref{prop:compact-gen} (to see that the generators are compact).
The result follows.
\end{proof}

With these preparations out of the way, we come to our main theorem.
We refer to \S\ref{subsec:summary} for a summary of the good $\tau$-self maps we managed to construct.

\begin{theorem} \label{thm:main}
Let $S$ be a scheme locally of finite dimension, $1 \le m, n \le \infty$.
Suppose $\tau: \1/(\ell^n,\rho^m) \to \1/(\ell^n,\rho^m)(r)$ is a $\tau$-self map.
Assume that $1/\ell \in S$ and for every $s \in S$ we have $\vcd_\ell(s) < \infty$.
Assume further that either $S$ is the spectrum of a field or $\tau$ is good.

Then for every $E \in \SH(S)_{\ell,\rho}^\comp$ the map \[ E/(\ell^n,\rho^m) \to E/(\ell^n,\rho^m)[\tau^{-1}]_{\ell,\rho}^\comp \] is an étale localization.
\end{theorem}
\begin{proof}
By Lemma \ref{lemm:bott-element-etale-invertible}, the map in question is an étale equivalence.
We thus need to prove that $E' := E/(\ell^n,\rho^m)[\tau^{-1}]_{\ell,\rho}^\comp$ is étale local.

Let $\scr X \to X \in \Sm_S$ be an étale hypercover and $K = \Sigma^\infty_+ cof(\scr X \to X)$.
We need to show that $[K, E'[i](j)] = 0$ for all $i, j$.
Since $E$ was arbitrary, replacing it by $E(i)[j]$ we may as well assume that $i=j=0$.
By the definition of goodness, see Definition \ref{def:tau-self-map}(3), any map $K \to E'$ factors through $K \to K'' := K/(\ell^{n'},\rho^{m'})[\tau^{-1}]_{\ell,\rho}^\comp$.
It is thus enough to show that $K''  = 0$.
It follows from \cite[Proposition B.3]{norms} that the collection of functors $\{s^*: \SH(S)_{\ell,\rho}^\comp \to \SH(s)_{\ell,\rho}^\comp \mid s \in S\}$ is conservative; 
also each of these functors preserves colimits.
It is thus enough to show that $s^*(K'') = 0$.
By construction $s^*(K)$ is étale-locally equivalent to zero; hence it suffices to show that $s^*(K'') \wequi s^*(K)/(\ell^{n'},\rho^{m'})[\tau^{-1}]_{\ell,\rho}^\comp$ is étale local.
We have thus reduced to the case of fields.

Assume now that $S = Spec(k)$.
Then $\1/(\ell^n,\rho^m)[\tau^{-1}]/(\ell,\rho)$ is étale local by Corollary \ref{cor:sphere-inversion}, 
and hence so is $\1/(\ell^n,\rho^m)[\tau^{-1}]_{p, \rho}^\comp$, being a limit of extensions of the former term.
Note that \[ E' \wequi E_{\ell,\rho}^\comp \wedge \1/(\ell^n,\rho^m)[\tau^{-1}] \in \SH(k)_{\ell,\rho}^\comp, \] with the smash product being formed in $\SH(k)_{\ell,\rho}^\comp$ (i.e., completed).
Since étale localization is smashing on $\SH(k)_{\ell,\rho}^\comp$ by Corollary \ref{cor:etale-locn-smashing}, we deduce that $E'$ is étale local in $\SH(k)_{\ell,\rho}^\comp$, 
and hence also in $\SH(k)$.
\end{proof}


\begin{remark} \label{rem:recoll}
Suppose that $\ell$ is not invertible on $S$.
Let $j: U = S[1/\ell] \hookrightarrow S$ and $Z = S \setminus U$ the closed complement.
Since $\SH_\et(Z)_\ell^\wedge = 0$ \cite[Theorem A.1]{bachmann-SHet2}
This implies that for $E \in \SH(S)$ we have 
\NB{$\SH_\et(\ph)$ satisfies localization, so $L_\et E/\ell^n$ can be glued from $i_!^\et i^*_\et L_\et E/\ell^n \wequi i_!^\et L_\et i^* E/\ell^n = 0$ 
and $j_*^\et j^*_\et L_\et E/\ell^n \wequi j_* L_\et j^*E/\ell^n$.} 
\[ L_\et E/\ell^n \wequi j_* L_\et j^*E/\ell^n. \]
Consequently the assumption that $1/\ell \in S$ in Theorem \ref{thm:main} is essentially harmless.
\end{remark}

We conclude with three immediate applications.

\begin{corollary} \label{cor:main}
Assumptions as in Theorem \ref{thm:main}.

\begin{enumerate}
\item Let $\ell$ be odd, $n < \infty$ and $m=\infty$.
  Then \[ E/\ell^n \to (E/\ell^n)^+[\tau^{-1}] \] is an étale localization.
\item Let $\ell=2$, $n < \infty$, $m=\infty$ and assume that $-1$ a square in $S$.
  Then \[ E/2^n \to E/2^n[\tau^{-1}] \] is an étale localization.
\item Suppose that $S$ is defined over a field containing a primitive $\ell$-th root of unity and satisfying $k^\times/\ell = \{1\}$.
  Then we have \[ \SH_\et(S)_\ell^\comp \wequi \SH(S)_\ell^\comp[\tau^{-1}]. \]
\end{enumerate}
\end{corollary}
\begin{proof}
In each case the $\rho$-completion is unnecessary (or simplified to $(\ph)^+$ in case (1)), by Remark \ref{rmk:rho-completion}.
In (1), (2) also the $\ell$-completion is unnecessary, by Remark \ref{rmk:p-completion-unnec}.
\end{proof}

Denote by $L_\et^\wedge: \SH(S) \to \SH_\et(S)_{\ell,\rho}^\wedge$ the left adjoint to the inclusion $\iota: \SH_\et(S)_{\ell,\rho}^\wedge \subset \SH(S)$.
\begin{corollary} \label{cor:integral}
Assumptions as in Theorem \ref{thm:main}.
Suppose that $m,n<\infty$.

Then the functor $\iota \circ L_\et^\wedge: \SH(S) \to \SH(S)$ is equivalent to Bousfield localization at the homology theory $\1/(\ell^n, \rho^m)[\tau^{-1}]$.
\end{corollary}
\begin{proof}
Let $E \in \SH(S)$ and $\alpha: E \to \iota L_\et^\wedge E$ be the $(\ell,\rho)$-complete étale localization map.
We first show that $\alpha$ is a $\1/(\ell^n, \rho^m)[\tau^{-1}]$-equivalence, or in other words that $\alpha \wedge\1/(\ell^n, \rho^m)[\tau^{-1}]$ is an equivalence.
By Theorem \ref{thm:main}, both sides are étale local, so it is enough to show that 
\[ L_\et(\alpha \wedge \1/(\ell^n, \rho^m)[\tau^{-1}]) \wequi L_\et(\alpha) \wedge L_\et(\1/(\ell^n, \rho^m)[\tau^{-1}]) \] is an equivalence.
This is clear.

It remains to show that $\SH_\et(S)_{\ell,\rho}^\wedge$ consists of $\1/(\ell^n,\rho^m)[\tau^{-1}]$-local objects.
In other words, if $E \in \SH(S)$ with $E \wedge \1/(\ell^n,\rho^m)[\tau^{-1}] = 0$, then we need to show that $\iota L_\et^\wedge(E) = 0$.
Again by Theorem \ref{thm:main} we have \[ \iota L_\et^\wedge(E)/(\ell^n,\rho^m) \wequi \iota L_\et^\wedge(E/(\ell^n,\rho^m)) \wequi E \wedge \1/(\ell^n,\rho^m)[\tau^{-1}] = 0. \]
Since $\iota L_\et^\wedge(E)$ is $(\ell, \rho)$-complete, we deduce that $\iota L_\et^\wedge(E) = 0$, as desired.
\end{proof}

\begin{corollary} \label{cor:smashing}
Assumptions as in Theorem \ref{thm:main}.
Then étale localization is smashing on $\SH(S)_{\ell,\rho}^\wedge$.
\end{corollary}
\begin{proof}
Denote by $L_\et: \SH(S)_{\ell,\rho}^\wedge \to \SH(S)_{\ell,\rho}^\wedge$ the étale localization functor.
It suffices to show that, for $E \in \SH(S)_{\ell,\rho}^\wedge$, the spectrum $E' := E \wedge L_\et(\1) \in \SH(S)_{\ell,\rho}^\wedge$ is étale local.
By definition $E'$ is $(\ell,\rho)$-complete, so $E' \wequi \lim_{n,m} E'/(\ell^n,\rho^m)$, and it suffices to show that $E'/(\ell^n,\rho^m)$ is étale local (for $n, m$ sufficiently large).
Theorem \ref{thm:main} implies that \[ E'/(\ell^n,\rho^m) \wequi E \wedge L_\et(\1/(\ell^n,\rho^m)) \wequi E \wedge \1/(\ell^n,\rho^m)[\tau_{n,m}^{-1}], \] 
where $\tau_{n,m}$ denotes a $\tau$-self map mod $(\ell^n,\rho^m)$ (it follows from our assumptions that this exists, at least for $n>1$).
Again by Theorem \ref{thm:main}, $E \wedge \1/(\ell^n,\rho^m)[\tau_{n,m}^{-1}]$ is étale local.
This concludes the proof.
\end{proof}

We also obtain the following new base change result, generalizing \cite[Theorem 7.5]{elso}.
\begin{theorem} \label{thm:basechange} Assumptions as in Theorem \ref{thm:main}. Suppose that $f:T \rightarrow S$ is a morphism of finite type. Then the functor
\[
f^*:\SH(S)_{\ell,\rho}^\wedge \rightarrow \SH(T)_{\ell,\rho}^\wedge,
\]
preserves \'etale local spectra. In particular the canonical map
\[
f^*(\1_{\et})_S \rightarrow (\1_{\et})_T
\]
is an equivalence in $\SH(T)_{\ell,\rho}^\wedge$.
\end{theorem}

\begin{proof} Follows from Theorem~\ref{thm:main} and the fact that the Bott elements are stable under base change by definition.
\end{proof}

\begin{remark} \label{rem:fracture}
If $S$ is finite dimensional, then étale localization is smashing on $\SH(S)_\Q$, since it just corresponds to ``passing to the plus part''\footnote{The minus part vanishes étale locally (see e.g. the proof of \cite[Theorem 7.2]{bachmann-SHet}), and the plus part has étale hyperdescent by \cite[Corollary 4.39]{clausen-mathew}: condition (1) holds by \cite[\S10.2]{norms} and for (2) we can take $A=H\Q$ (this is where we need to be in the plus part).}, or equivalently $\rho$-completion.
In particular Corollary \ref{cor:smashing} holds rationally (instead of completed at $\ell$) as well.
Using Remark~\ref{rem:recoll}, we see that it also holds completed at $\ell$ which fail to be invertible on $S$ (provided that $\vcd_p(s) < \infty$ for all $s \in S$).
One may deduce that étale localization is smashing on $\SH(S)_{\rho}^\comp$, provided that $\vcd(s) < \infty$ for all $s \in S$ (in addition to $S$ being finite dimensional).
Similarly \[ f^*:\SH(S)_{\rho}^\wedge \rightarrow \SH(T)_{\rho}^\wedge \] preserves \'etale local spectra (under the given assumption).
\end{remark}

\appendix

\section{Multiplicative structures on Moore objects}
\label{sec:moore-mult}
\subsection{Definitions and setup} \label{sec:def-set}
Let $\scr C$ be a semiadditive \cite[Definition 6.1.6.13]{HA} symmetric monoidal $\infty$-category.
For $n \in \N$ we call 
\[
\Mp{n} = \cof(\1 \xrightarrow{n} \1)
\] the \emph{mod $n$ Moore object}, which by construction comes with a map $u: \1 \to \Mp{n}$ which we call the \emph{unit}.
By a \emph{(homotopy) multiplication} on $\Mp{n}$ we mean a map 
\[
m: \Mp{n} \otimes \Mp{n} \to \Mp{n} \in \h\scr C,
\] which is compatible with the unit in the sense that \[ m \circ (\id \otimes u) = \id = m \circ (u \otimes \id): \Mp{n} \to \Mp{n} \in \h\scr C. \]

\begin{lemma}
If $m$ is a multiplication on $\Mp{n}$, then $u: \1 \to \Mp{n}$ is a morphism of homotopy ring objects.
\end{lemma}
\begin{proof}
Consider the diagram
\begin{equation*}
\begin{CD}
\1 \otimes \1 @>{\id \otimes u}>> \1 \otimes \Mp{n} @>{u \otimes \id}>> \Mp{n} \otimes \Mp{n} \\
@V{m_0}VV                          @VVV                                 @VmVV                 \\
\1            @>u>>               \Mp{n}            @=                  \Mp{n}.
\end{CD}
\end{equation*}
Here, $m_0$ and the unlabelled map are canonical isomorphisms; at the same time $m_0$ is the multiplication in the canonical ring structure on $\1$.
Since the top composite is $u \otimes u$, we need to show that the outer rectangle commutes.
The left hand square commutes by definition of a symmetric monoidal category, and the right hand square commutes by assumption.
\end{proof}

If $\scr C$ is stable, then we have the cofiber sequence \[ \1 \xrightarrow{n} \1 \xrightarrow{u} \Mp{n} \xrightarrow{\partial} \1[1]. \]
The composite \[\delta = u \circ \partial: \Mp{n} \to \Mp{n}[1] \] is called the \emph{coboundary}.
A multiplication $m$ on $\Mp{n}$ is called \emph{regular} if $\delta$ is a derivation \cite[Definition 1]{oka1984multiplications}, in the sense that: \[ \delta \circ m \wequi m \circ (\id \otimes \delta + \delta \otimes \id). \]

\subsection{Review of Oka's results} \label{sec:oka}

The case $\scr C = \SH$ has been treated by Oka. Among other things the following are shown in \cite[Theorem 2]{oka1984multiplications}.
\begin{itemize}
\item $\Mp{n}$ has a multiplication if and only if $n \not\equiv 2 \pmod{4}$; the multiplication is unique if $n$ is odd.
\item Whenever $\Mp{n}$ has a multiplication, it may be chosen to be regular.
\item A multiplication on $\Mp{n}$ is (homotopy) commutative if and only if $n \equiv 0 \pmod{8}$ or $n$ is odd; it is (homotopy) associative if and only if $n \not\equiv 2 \pmod{4}$ and $n \not\equiv \pm 3 \pmod{9}$.
\item For any multiplication on any $\1/n$, the commutator \[m \circ (\id - \text{switch}): \Mp{n} \wedge \Mp{n} \to \Mp{n} \] factors through \[ \partial \wedge \partial: \Mp{n} \wedge \Mp{n} \to \1[1] \wedge \1[1]. \]
  Similarly, the associator \[ m \circ (m \wedge \id - \id \wedge m): \Mp{n}^{\wedge 3} \to \Mp{n} \] factors through\footnote{These statements are only interesting if the multiplication is not commutative/associative.} \[ \partial^{\wedge 3}: \Mp{n}^{\wedge 3} \to \1[1]^{\wedge 3}. \]
\end{itemize}

\subsection{Asymptotic unicity and associativity}
For a small idempotent complete $\infty$-category $\scr C$ with finite limits, let $\Pro(\scr C) = \Ind(\scr C^\op)^\op$ denote the category of \emph{pro-objects in $\scr C$} \cite[\S A.8.1, Remark A.8.1.2]{SAG}.
If $\scr C$ has a symmetric monoidal structure, then so does $\Pro(\scr C)$ \cite[Remark 2.4.2.7, Proposition 4.8.1.10]{HA}.
\begin{lemma} \label{lemm:smashing-loc}
Let $\scr C$ be a small stable idempotent complete symmetric monoidal $\infty$-category $\scr C$.
Then $\ell$-completion is a smashing localization on $\Pro(\scr C)$.
\end{lemma}
\begin{remark} \label{rem:compl-formula}
Even though $\Pro(\scr C)$ is not presentable, $\ell$-completion (i.e., localization at the $\ell$-equivalences) exists: it is clear that $E \mapsto \lim_n E/\ell^n$ is a $\ell$-equivalence to a $\ell$-complete object.
\end{remark}
\begin{proof}
From the formula given in Remark~\ref{rem:compl-formula}, this is the case in any symmetric monoidal stable $\infty$-category where the tensor product commutes with cofiltered limits, such as $\Pro(\scr C)$. 
\end{proof}

Let \[ c: \SH^\omega \hookrightarrow \Pro(\SH^\omega) \] denote the inclusion ``at constant cofiltered systems''.
\begin{proposition} \label{prop:pro-moore}
Let $\mu_i: \1/\ell^i \wedge \1/\ell^i \to \1/\ell^i$ be a sequence of unital multiplication maps in $\SH$ 
\NB{left-unital multiplications seem to be enough?} 
(for $i \ge 2$ if $\ell=2$), such that the reductions $r: \1/\ell^{i+1} \to \1/\ell^i$ are multiplicative (up to homotopy).
Then the following hold:
\begin{enumerate}
\item The pro-spectrum $c(\1)_\ell^\comp$ is represented by the inverse system \[ \dots \xrightarrow{r} \1/\ell^4 \xrightarrow{r} \1/\ell^3 \xrightarrow{r} \1/\ell^2. \]
\item The multiplication on $c(\1)_\ell^\comp$ is represented by the system of maps \[ c(\1)_\ell^\comp \wedge c(\1)_\ell^\comp \wequi \{(\1/\ell^i)^{\wedge 2}\}_i \xrightarrow{\mu_i} \{\1/\ell^i\}_i. \]
\end{enumerate}
\end{proposition}
\begin{proof}
(1) Clear by the formula for $\ell$-completion in any stable category; see Remark~\ref{rem:compl-formula}.

(2) Since $\ell$-completion is smashing in $\Pro(\SH^\omega)$ by Lemma \ref{lemm:smashing-loc}, the multiplication map is inverse to the unit map (on either side), and hence it is enough to show that the following composite map of pro-systems is homotopic to the identity
\[ \alpha: \{\1/\ell^n\}_n \wequi \{\1 \wedge \1/\ell^n\}_n \xrightarrow{r \wedge \id} \{\1/\ell^n \wedge \1/\ell^n\}_n \xrightarrow{\mu_n} \{\1/\ell^n\}_n. \]
Using the formula for mapping spaces in categories of pro-objects and the Milnor exact sequence, we obtain an exact sequence \[ 0 \to \mathrm{lim}^1_m \colim_n [\Sigma \1/\ell^n, \1/\ell^m] \to [\{\1/\ell^n\}_n, \{\1/\ell^m\}_m]_{\Pro(\SH^\omega)} \to \lim_m \colim_n [\1/\ell^n, \1/\ell^m] \to 0. \]
By assumption, $\alpha$ corresponds to the identity in the right hand group; it thus suffices to show that the $\mathrm{lim}^1$-term vanishes.
The (strong) dual of $\1/\ell^n$ is $\Sigma^{-1}\1/\ell^n$, and multiplication by $\ell^n$ is zero on $\1/\ell^m$ for $n$ sufficiently large.
It follows that \[ \colim_n [\Sigma \1/\ell^n, \1/\ell^m] \wequi \pi_1(\1/\ell^m) \oplus \pi_2(\1/\ell^m). \]
These groups are all finite (in fact independent of $m$ for $m \gg 0$) and hence the inverse system is Mittag--Leffler, and $\mathrm{lim}^1 = 0$ as desired.
\end{proof}

\begin{corollary} \label{cor:asymp}
Let $\{\mu_i\}$ be as in Proposition \ref{prop:pro-moore}.
Then for $n \gg 0$, the Moore object $\1/\ell$ can be given the structure of a homotopy associative $\1/\ell^n$-module.
\end{corollary}
\begin{proof}
Since $c(\1)_\ell^\comp$ is a commutative monoid, $M=c(\1)_\ell^\comp/\ell \wequi c(\1)/\ell$ is an $\scr A_\infty$-module under $c(\1)_\ell^\comp$.
Note that $M$ is represented by the constant pro-spectrum $\{\1/\ell\}_n$.
The multiplication map $c(\1)_\ell^\comp \wedge M \to M$ thus yields (compatible) multiplication maps $\1/\ell^n \wedge \1/\ell \to \1/\ell$, for $n$ sufficiently large.
The associator $c(\1)_\ell^\comp \wedge c(\1)_\ell^\comp \wedge M \to M$ is zero as a map of pro-objects (the module structure being $\scr A_\infty$). Hence the associator $\1/\ell^n \wedge \1/\ell^n \wedge \1/\ell \to \1/\ell$ is zero for $n$ sufficiently large. Unitality of the multiplication is handled similarly.
\end{proof}

\begin{remark}
Oka constructs similar module structures by hand (and for explicit $n$) \cite[Section 6]{oka1984multiplications}, without addressing associativity.
\end{remark}

\subsection{Uniqueness for $H\Z$}

\begin{lemma} \label{lemm:uniqueness-mult}
Let $n \ge 1$.
\begin{enumerate}
\item The Moore object $\Z/n \in D(\Z)$ admits a unique (up to homotopy) multiplication.
\item Let $m \ge 1$.
  The Moore object $\Z/n \in D(\Z)$ admits a unique (up to homotopy) structure of $\Z/mn$-module such that the canonical map $\Z/mn \to \Z/n$ is a module map.
\end{enumerate}
\end{lemma}
\begin{proof}
For $A, B, C \in D(\Z)^\heartsuit$, we have \[ [A \otimes B, C] \wequi [\pi_0(A \otimes B), C]. \]
This implies that a multiplication on $\Z/n$ is the same as a multiplication in the category of abelian groups, and similarly for the module structure of (2).
But then uniqueness follows from the fact that $\Z \to \Z/mn \to \Z/n$ are surjections (and existence is even more obvious).
\end{proof}

Equivalently, the Eilenberg--MacLane spectrum $H\Z/n \in \SH$ admits a unique $H\Z$-linear multiplication compatible with the canonical unit map.

\begin{corollary} \label{cor:multn-HZ-correct}
Let $n \ge 1$ and choose a multiplication $m: \Mp{n} \wedge \Mp{n} \to \Mp{n} \in \SH$.
Let $e: \SH \to \SH(S)$ denote the ``constant sheaf'' functor.
Then the induced multiplication on $e(\Mp{n}) \wedge H\Z \wequi H\Z/n$ is the canonical one.
Similarly for a choice of $\Mp{mn}$-module structure on $\Mp{n}$.
\end{corollary}
\begin{proof}
Consider the commutative diagram
\begin{equation*}
\begin{CD}
\SH @>e>> \SH(S) \\
@VMVV      @VVMV   \\
D(\Z) @>e>> \DM(S),
\end{CD}
\end{equation*}
where the left (respectively right) vertical functors is given by tensoring with the Eilenberg--MacLane spectrum (respectively motivic Eilenberg--MacLane spectrum). Along the left vertical arrow, by Lemma \ref{lemm:uniqueness-mult}, we see the multiplication on $M(\1/n)$ induced by the choice of multiplication on $\1/n$ is the same as the standard multiplication on $M(\1/n) = \Z/n$.
It follows that the induced multiplication on $M(e(\1/n)) \wequi e(M(\1/n))$ is also the standard one.
Let $U: \DM(S) \to \SH(S)$ denote the right adjoint to $M$; it follows that $UMe(\1/n) \wequi U(\1) \wedge \1/n \wequi H\Z/n$ carries the correct multiplication. The same argument works for module structures.

\end{proof}

\section{Inverting elements in homotopy rings} \label{sec:invert}
A variant of the following result is stated without proof on page 1 of \cite{arthan1983localization}.
\begin{lemma} \label{lem:invert}
Let $\scr C$ be a symmetric monoidal $\infty$-category in which $\N$-indexed colimits exits and are preserved by $\otimes$ in each variable separately.
Let $E, L \in \scr C$ and suppose given a map $\tau: E \to E \otimes L$.
\begin{enumerate}
\item There exists a canonical diagram $F: \N \to \scr C$, informally described as \[ E \xrightarrow{\tau} E \otimes L \xrightarrow{\tau \otimes \id_L} E \otimes L^{\otimes 2} \to \cdots. \]
  Denote its colimit by $E[\tau^{-1}]$.

\item Given a map $u: \1 \to E$, there is a canonically induced map $\bar{u}: \1 \to E[\tau^{-1}]$.

\item Given a map $m: E \otimes E \to E$ and a homotopy $h$ as in the following diagram
\begin{equation*}
\begin{tikzcd}
E \otimes E \ar[r, "\tau \otimes \tau"] \ar[d, "m"] & E \otimes L \otimes E
\otimes L \ar[r, "\wequi"] & E \otimes E \otimes L \otimes L \ar[d, "m \otimes \id_{L} \otimes \id_L"] \\
E \ar[r, "\tau"'] \ar[urr, Leftrightarrow, "h"] & E \otimes L \ar[r, "\tau \otimes \id_L"] & E \otimes L \otimes L,
\end{tikzcd}
\end{equation*}
  there is a canonically induced map $\bar{m}: E[\tau^{-1}] \otimes E[\tau^{-1}] \to E[\tau^{-1}]$.

\item Suppose that $\scr C$ is stable and compactly generated, and $L$ invertible.
  Suppose furthermore that the diagrams 
\begin{equation*}
\begin{tikzcd}
E \otimes \1 \ar[r, "\id \otimes u"] & E \otimes E \ar[d, "m"] \quad & E \otimes \1 \ar[r, "\id \otimes \tau u"] & E \otimes E \otimes L \ar[d, "m"] \\
E \ar[u, "\wequi"] \ar[r, "\id"]                   & E               & E \ar[u, "\wequi"] \ar[r, "\tau"]                   & E \otimes L,
\end{tikzcd}
\end{equation*}
commute in $h\scr C$. (In other words, the multiplication $m$ is homotopy left unital, and $\tau$ is given by left multiplication by a certain homotopy element.) Then the composite \[ E[\tau^{-1}] \wequi E[\tau^{-1}] \otimes \1 \xrightarrow{\id \otimes \bar{u}} E[\tau^{-1}] \otimes E[\tau^{-1}] \xrightarrow{\bar{m}} E[\tau^{-1}] \] is an equivalence.
(A similar result holds for left multiplication.)
\end{enumerate}
\end{lemma}
\begin{proof}
(1) Let $\N'$ be the simplicial set with $0$-cells $0, 1, 2, \cdots$ and for each $n \in \N$ a $1$-cell from $n$ to $n+1$.
There is a canonical map of simplicial sets $\N' \to N\N$ (here we use $N$ to indicate the nerve of a category, for once not silently identifying categories with their nerves) which is an inner anodyne extension\footnote{Let $X_n = \Delta^{0,1} \coprod_{1} \Delta^{1,2} \coprod_{2} \dots \coprod_{n-1} \Delta^{n-1,n}$, so that $X_n \to \Delta^n$ is inner anodyne (see e.g., \cite[Proof of Proposition 3.2.1.13]{HTT}). Let $Y_n = \Delta^n \coprod_{X_n} \N'$. Since $\N' \wequi \colim_n X_n$ and $N\N \wequi \colim_n \Delta^n$, we have $\colim_n Y_n \wequi N\N$. Since the class of inner anodyne maps is weakly saturated (see \cite[Definition A.1.2.2]{HTT}), each map $\N' \to Y_n$ is inner anodyne, and hence so is the colimiting map $\N' \to N\N$.} and hence in particular a categorical equivalence \cite[Lemma 2.2.5.2]{HTT}.
It follows that $\Fun(\N, \scr C) \to \Fun(\N', \scr C)$ is an equivalence.
The endomorphism $\tau$ induces an element of $\Fun(\N', \scr C)$ as displayed, which hence canonically lifts as claimed.

(2) $\bar{u}$ exists by definition of a colimit diagram.
We can make it slightly more explicit as follows.
Note that $\N' \times \Delta^1 \to N\N \times \Delta^1$ is also a categorical equivalence \cite[Corollary 2.2.5.4]{HTT}, and consequently maps in $\Fun(\N, \scr C)$ can be produced as the evident ladder diagrams, with homotopies filling the squares.
We now consider the morphism of diagrams
\begin{equation*}
\begin{tikzcd}
\1 \ar[r, "\wequi"] \ar[d, "u"] &  \1 \ar[r, "\wequi"] \ar[d, "\tau u"] & \1 \ar[r, "\wequi"] \ar[d, "\tau^2 u"] & \cdots \\
E \ar[r, "\tau"] \ar[ur, Leftrightarrow] &E \otimes L \ar[r, "\tau \otimes
\id_L"] \ar[ur, Leftrightarrow] & E \otimes L \otimes L \ar [r] \ar[ur, Leftrightarrow] & \cdots,
\end{tikzcd}
\end{equation*}
where the homotopies are the tautological ones.
The induces map on colimits is $\bar{u}$.

(3) Since $\otimes$ in $\scr C$ preserves $\N$-indexed colimits in each variable separately (by assumption), we have 
\[
E[\tau^{-1}] \otimes E[\tau^{-1}] \wequi \colim_{\N \times \N} F \otimes F.
\]
The diagonal $\Delta: \N \to \N \times \N$ is cofinal and hence 
\[
E[\tau^{-1}] \otimes E[\tau^{-1}] \wequi \colim_\N F\otimes F \circ \Delta \wequi \colim_{\N'} F\otimes F \circ \Delta.
\]
The map $2: \N \to \N$ is also cofinal. Consequently in order to produce $\bar{m}$, it suffices to produce homotopies in the following diagram
\begin{equation*}
\begin{tikzcd}
E \otimes E \ar[r, "\tau \otimes \tau"] \ar[d, "m_0"] & E \otimes L \otimes E \otimes L \ar[r, "\tau \otimes \tau"] \ar[d, "m_1"] & E \otimes L^{\otimes 2} \otimes E \otimes L^{\otimes 2} \ar[r] \ar[d, "m_2"] & \cdots \\
E \ar[r, "\tau^2"] \ar[ur, Leftrightarrow, "h_0"] & E \otimes L^{\otimes 2} \ar[r, "\tau^2"] \ar[ur, Leftrightarrow, "h_1"] & E \otimes L^{\otimes 4} \ar[r] \ar[ur, Leftrightarrow, "h_2"] & \cdots.
\end{tikzcd}
\end{equation*}
Here $m_0 = m$ and $m_i$ for $i>0$ is obtained as the evident composite of $m$ and switch maps.
We are provided with $h=h_0$. If we denote the $i$-th square above as $S_i$, then there is a canonical equivalence $S_i \wequi S_0 \otimes \id_{L^{\otimes 2i}}$. We may thus choose $h_i = h \otimes \id_{L^{\otimes 2i}}$.

(4) As before we have $E[\tau^{-1}] \otimes \1 \wequi \colim_{\N \times \N} F \otimes \1 \wequi \colim_{\N'} F \otimes \1 \circ \Delta$.
The composite in question is the colimit of the following (composed ladder) diagram
\begin{equation*}
\begin{CD}
\begin{tikzcd}
E \otimes \1 \ar[r, "\tau \otimes \id"] \ar[d, "\id \otimes u"] &  E \otimes L \otimes \1 \ar[r, "\tau \otimes \id \otimes \id"] \ar[d, "\id \otimes \id \otimes \tau u"] & E \otimes L^{\otimes 2} \otimes \1 \ar[r] \ar[d, "\id \otimes \id \otimes \tau^2 u"] & \cdots \\
E \otimes E \ar[r, "\tau \otimes \tau"] \ar[d, "m"] & E \otimes L \otimes E \otimes L \ar[r, "\tau \otimes \tau"] \ar[d, "m"] & E \otimes L^{\otimes 2} \otimes E \otimes L^{\otimes 2} \ar[r] \ar[d, "m"] & \cdots \\
E \ar[r, "\tau^2"] & E \otimes L^{\otimes 2} \ar[r, "\tau^2"] & E \otimes L^{\otimes 4} \ar[r] & \cdots.
\end{tikzcd}
\end{CD}
\end{equation*}
Here, we have suppressed the homotopies and labelled all the multiplication maps $m$, even if a switch map is involved. By assumption, the vertical maps are homotopic to a composite of $\tau$'s and switch maps. 

Now, let $T \in \scr C$ be compact and denote by $\pi: \scr C \to \mathrm{Set}$ the functor $[T, \ph]$.
Then $\pi$ preserves filtered colimits and the collection of all such functors $\pi$ is conservative (by assumption).
Hence it suffices to show that when applying $\pi$ to the above morphism of diagrams, we obtain an isomorphism of ind-objects in sets.
Restricting to the cofinal subcategory $\{2, 4, 8, 16, \dots \}$ of $\N$, we obtain a diagram
\begin{equation*}
\begin{tikzcd}
\pi(E \otimes L^{\otimes 2}) \ar[r, "\tau^2"] \ar[d, "\tilde{\tau}^2"] & \pi(E \otimes L^{\otimes 4}) \ar[r, "\tau^4"] \ar[d, "\tilde{\tau}^4"] & \pi(E \otimes L^{\otimes 8}) \ar[r, "\tau^8"] \ar[d, "\tilde{\tau}^8"] & \pi(E \otimes L^{\otimes 16}R \ar[r] \ar[d, "\tilde{\tau}^{16}"] & \cdots \\
\pi(E \otimes L^{\otimes 4}) \ar[r, "\tau^4"] \ar[ur, dashrightarrow, "\id"] & \pi(E \otimes L^{\otimes 8}) \ar[r, "\tau^8"] \ar[ur, dashrightarrow, "\id"] & \pi(E \otimes L^{\otimes 16}) \ar[r, "\tau^{16}"] \ar[ur, dashrightarrow, "\id"] & \pi(E \otimes L^{\otimes 32}) \ar[r] \ar[ur, dashrightarrow, "\id"] & \cdots,
\end{tikzcd}
\end{equation*}
where $\tilde{\tau}^n$ denotes an appropriate composite of $\tau$ and switch maps.
We claim that the dashed identity maps define an inverse morphism of ind-objects. Since we are working in a 1-category, we need only verify that all triangles commute.
Since $L$ is invertible, the switch map on $L^2$ is homotopic to the identity \cite[Lemma 4.19]{dugger2014coherence}, and hence $\tilde{\tau}^{2n} = \tau^{2n}$.
This proves the desired commutativity.

\end{proof}

Hence, under the assumption that $L$ is an invertible object in $\scr C$ and $\scr C$ is compactly generated (these will hold for all our applications), we conclude that $E[\tau^{-1}]$ admits ``homotopy multiplication'' which is not quite, but close to, left unital. We need the next variant for modules.

\begin{lemma} \label{lem:modules} In the situation of Lemma~\ref{lem:invert}, further assume that we are given an object $M \in \scr C$ equipped with a map $a: E \otimes M \rightarrow M$. 
\begin{enumerate}
\item In the situation of Lemma~\ref{lem:invert}(2), consider the map
\[
\tau \cdot: M \xrightarrow{u \otimes \id} E \otimes M \xrightarrow{\tau \otimes \id} E \otimes L \otimes M \simeq L \otimes E \otimes M \xrightarrow{\id \otimes a} L \otimes M.
\]
Then there is a diagram $F_M: \N \to \scr C$ informally described as 
\[
M \xrightarrow{\tau \cdot} L \otimes M \xrightarrow{\id \otimes \tau \cdot} L \otimes L \otimes M \xrightarrow{\id \otimes (\tau \cdot)^{\otimes 2}} \cdots.
\]
Denote its colimit by $M[\tau^{-1}]$. 
\item In the situation of Lemma~\ref{lem:invert}(3), suppose further that there we are given a homotopy $h_M$:
\[
\begin{tikzcd}
E \otimes M \ar[r, "\tau \otimes \tau \cdot "] \ar[d, "\id \otimes a"] & E \otimes L \otimes L \otimes M  \ar[r, "\wequi"] &  L \otimes L \otimes E \otimes M \ar[d, "\id \otimes a"] \\
M \ar[r, "\tau \cdot "'] \ar[urr, Leftrightarrow, "h_M"] & L \otimes M \ar[r, "\tau \cdot "] & L \otimes L  \otimes M.
\end{tikzcd} 
\]
Then, there is canonically induced map
\[
\bar{a}: E[\tau^{-1}] \otimes M[\tau^{-1}] \rightarrow M[\tau^{-1}],
\]
such that the diagram
\[
\begin{tikzcd}
E \otimes M \ar{r}{a} \ar{d} & M \ar{d}\\
E[\tau^{-1}] \otimes M[\tau^{-1}] \ar{r}{\bar{a}} & M[\tau^{-1}]
\end{tikzcd}
\]
commutes.
\item Assume that we are in the situation of Lemma~\ref{lem:invert}(4) and further assume that the diagram
\[
\begin{tikzcd}
\1 \otimes M \ar[r, "u \otimes \id"] & E \otimes M \ar[d, "a"] \quad \\
M \ar[u, "\wequi"] \ar[r, "\id"]                  & M,
\end{tikzcd}
\]
commutes. Then the composite
 \[ M[\tau^{-1}] \wequi  \1 \otimes M[\tau^{-1}] \xrightarrow{\bar{u} \otimes \id} E[\tau^{-1}] \otimes M[\tau^{-1}] \xrightarrow{\bar{a}} M[\tau^{-1}] \] is an equivalence.
\end{enumerate}

\end{lemma}

\begin{proof}
(1) By the same reasoning as in Lemma~\ref{lem:invert}(1), the described diagram exists.

(2)  By the same reasoning as in Lemma~\ref{lem:invert}(3), it suffices produce a transformation
\[
F \otimes F_M \circ \Delta \Rightarrow F_M.
\]
This is displayed in the following diagram:
\begin{equation*}
\begin{tikzcd}
E \otimes M \ar[r, "\tau \otimes \tau"] \ar[d, "a_0"] & E \otimes L \otimes L \otimes M \ar[r, "\tau \otimes \tau"] \ar[d, "a_1"] & E \otimes L^{\otimes 2} \otimes L^{\otimes 2} \otimes M \ar[r] \ar[d, "a_2"] & \cdots \\
M \ar[r, "\tau^2\cdot"] \ar[ur, Leftrightarrow, "h_0"] & L^{\otimes 2} \otimes M \ar[r, "\tau^2 \cdot"] \ar[ur, Leftrightarrow, "h_1"] & L^{\otimes 4} \otimes M \ar[r] \ar[ur, Leftrightarrow, "h_2"] & \cdots.
\end{tikzcd}
\end{equation*}
Here, $a_0 = a$ and $a_i$ is a composite of the evident switch maps and $a$. The homotopy $h_0$ is $h_M$, while $h_i:= h_M \otimes \id_{L^{\otimes 2i}}$.

(3) By the same reasoning as in Lemma~\ref{lem:invert}(4), the composite is given by taking colimits:
\begin{equation*}
\begin{CD}
\begin{tikzcd}
\1 \otimes M  \ar[r, "\id \otimes \tau \cdot"] \ar[d, "u \otimes \id"] & \1 \otimes L \otimes M \ar[r, "\id \otimes \tau \cdot"] \ar[d, "\tau u \otimes \id \otimes \id"] & \1 \otimes L^{\otimes 2} \otimes  M \ar[r] \ar[d, "\tau^2 u \otimes \id \otimes \id"] & \cdots \\
E  \otimes M \ar[r, "\tau \otimes \tau"] \ar[d, "a"] & E \otimes L \otimes L \otimes M \ar[r, "\tau \otimes \tau"] \ar[d, "a"] & E \otimes L^{\otimes 2} \otimes L^{\otimes 2} \ar[r] \ar[d, "\id \otimes a"] \otimes M & \cdots \\
M \ar[r, "\tau^2"] & L^{\otimes 2} \otimes M \ar[r, "\tau^2"] &  L^{\otimes 4} \otimes M \ar[r] & \cdots,
\end{tikzcd}
\end{CD}
\end{equation*}
where we have suppressed the switch maps. By assumption, the vertical composites are homotopic to a composite of $\tau$'s and switch maps. It remains to note that the following diagram in sets

\begin{equation*}
\begin{tikzcd}
\pi(L^{\otimes 2} \otimes M) \ar[r, "\tau^2\cdot"] \ar[d, "\tilde{\tau}^2"] & \pi(L^{\otimes 4} \otimes M) \ar[r, "\tau^4\cdot"] \ar[d, "\tilde{\tau}^4\cdot"] & \pi(L^{\otimes 8} \otimes M) \ar[r, "\tau^8\cdot"] \ar[d, "\tilde{\tau}^8\cdot"] & \pi(L^{\otimes 16} \otimes M) \ar[r] \ar[d, "\tilde{\tau}^{16}\cdot"] & \cdots \\
\pi(L^{\otimes 4} \otimes M) \ar[r, "\tau^4\cdot"] \ar[ur, dashrightarrow, "\id"] & \pi(L^{\otimes 8} \otimes M) \ar[r, "\tau^8\cdot"] \ar[ur, dashrightarrow, "\id"] & \pi(L^{\otimes 16} \otimes M) \ar[r, "\tau^{16}\cdot"] \ar[ur, dashrightarrow, "\id"] & \pi(L^{\otimes 32} \otimes M) \ar[r] \ar[ur, dashrightarrow, "\id"] & \cdots
\end{tikzcd}
\end{equation*}
is commutative. This follows by the same reason as in Lemma~\ref{lem:invert}(4) --- the switch maps are homotopic to the identity.
\end{proof}

\begin{corollary} \label{cor:factor-through-free}
Let $\scr C$ be a compactly generated, stable, presentably symmetric monoidal $\infty$-category, $E \in \scr C$ a homotopy unital associative ring, $L \in Pic(\scr C)$ and $\tau: \1 \to E \otimes L$ a homotopy central element (i.e., left and right multiplication by $\tau$ induce homotopic maps $E \to E \otimes L$).
Let $M$ be a homotopy associative $E$-module (for example, $E=M$).

Then for any $X, Y \in \scr C$, any map $X \to Y \otimes M[\tau^{-1}]$ factors through $X \otimes E[\tau^{-1}]$.
\end{corollary}
\begin{proof}
We may apply the above lemmas (condition (2) is where we use centrality of $\tau$), and hence obtain
\[
\bar{u}: \1 \to E[\tau^{-1}], \bar{m}: E[\tau^{-1}] \otimes M[\tau^{-1}] \to M[\tau^{-1}]
\] such that ``right'' multiplication by $\bar{u}$ is an equivalence 
\[
\alpha: M[\tau^{-1}] \simeq M[\tau^{-1}]\otimes \1 \xrightarrow{\id \otimes \bar{u}} M[\tau^{-1}] \otimes E[\tau^{-1}] \xrightarrow{\bar{m}'}  M[\tau^{-1}].
\]
Here, $\bar{m}'$ is $\bar m$ composed with the switch map.
Let $f: X \to Y \otimes E[\tau^{-1}]$ be any map and consider the commutative diagram
\begin{equation*}
\begin{tikzcd}
X \otimes E[\tau^{-1}] \ar[r, "f \otimes \id"] & Y \otimes M[\tau^{-1}] \otimes E[\tau^{-1}] \ar[rd, "\id \otimes \bar{m}'"] \\
X \otimes \1 \ar[u, "\id \otimes \bar{u}"] \ar[r, "f \otimes \id"'] & Y \otimes M[\tau^{-1}] \otimes \1 \ar[u, "\id \otimes \id \otimes \bar{u}"] \ar[r, "\id \otimes \alpha"'] & Y \otimes M[\tau^{-1}] \ar[r, "\id \otimes \alpha^{-1}"'] & Y \otimes M[\tau^{-1}].
\end{tikzcd}
\end{equation*}
Since bottom composite is $f$, the result follows.
\end{proof}

\section{Galois descent in invertible characteristic} \label{sec:galois}
In this section we establish some folklore results to the effect that homotopy orbits and fixed points are ``the same'' and ``non-homotopical'' if the group order is invertible.
\NB{Surely there must be references for this?}\NB{These results are much stronger than what we actually use. Maybe just prove the homotopy category version?}
Then we apply this to obtain some essentially trivial Galois descent results for $\SH$ and $\DM$.
They are well-known to experts and generalize \cite[Appendix C]{elso} (via a different approach).

\subsection{Universality of homotopy fixed points}
Let $G$ be a finite group. Let $\scr C$ be an $\infty$-category with finite coproducts and finite products. Denote by $i: * \to BG$ the canonical map.
The functor $i^*: \Fun(BG, \scr C) \to \scr C$ has a left adjoint $i_\#$ and a right adjoint $i_*$, given by \[ i_\# X = \coprod_{g \in G} X \text{ and } i_* X = \prod_{g \in G} X. \]
If $\scr C$ is furthermore pointed, i.e., $(-1)$-semiadditive, there is a natural transformation $\mu: i_\# \to i_*$.
Equivalently, there is a natural transformation $i^* i_\# \to \id$; it is even natural in $\scr C$ \cite[Construction 3.13, Observation 3.14]{harpaz2017ambidexterity}.
In particular, if $F: \scr C \to \scr D$ is a functor of pointed $\infty$-categories with finite coproducts and finite products, preserving finite coproducts, and $A \in \scr C$, then the following diagram commutes canonically 
\begin{equation*}
\begin{CD}
F i_\# A @>{F \mu_A}>> F i_* A \\
@|                      @VVV   \\
i_\# F A @>{\mu_{FA}}>> i_* FA.
\end{CD}
\end{equation*}

Now suppose that $\scr C$ is semiadditive (i.e., $0$-semiadditive); then in particular $\mu_A$ is an equivalence for all $A \in \scr C$.
Denote by $\alpha: i_\# \to i_\#$ the natural transformation corresponding by adjunction to the diagonal \[
\id \to i^* i_\# \wequi \coprod_G \id \wequi \prod_G \id.
\]
Inverting the middle arrow in the following span \[ \id \xrightarrow{\eta} i_* i^* \xleftarrow{\mu i^*} i_\# i^* \xrightarrow{\alpha i^*} i_\# i^* \xrightarrow{\epsilon} \id \in \Fun(\Fun(BG, \scr C), \Fun(BG, \scr C)) \] yields a natural transformation $T: \id \to \id$ of the identity endofunctor of $\Fun(BG, \scr C)$.
This construction is natural in the semiadditive $\infty$-category $\scr C$.

\begin{remark} \label{rmk:T-reformulation}
For $A \in \Fun(BG, \scr C)$, the object $i^*A$ has a canonical endomorphism $t_A: i^*A \to i^*A$ given by $\sum_g g$.
In the notation of \cite[Definition 2.11]{carmeli2018ambidexterity} we have $T_A \wequi \int_i t_A$.
\end{remark}

\begin{example}
If $\scr C$ is the category of abelian groups, then $\Fun(BG, \scr C)$ is the category of abelian groups $A$ with an action by $G$, and $T_A: A \to A$ is given by $a \mapsto |G|\sum_{g \in G} ga$.
\tombubble{This seems somewhat stupid. We should be able to construct $\sum_g ga$. 
The problem is that this version is naturally constructed using $p_\#, p_*$ for $p: BG \to *$, 
and we do not want to assume at the outset that our functors preserve $p_*, p_\#$.}
\end{example}

\begin{lemma} \label{lemm:G-squaring}
Let $G$ have order $n$.
Then $T^2 \wequi n^2T$ as endo-transformations of the identity endofunctor of $\Fun(BG, \scr C)$.
\end{lemma}
\begin{proof}
This follows from \cite[Corollary 3.1.14]{carmeli2018ambidexterity}.
\end{proof}

\begin{definition}
For a semiadditive $\infty$-category $\scr C$ and a finite group $G$, denote by 
\[
\Fun(BG, \scr C)^{0} \subset \Fun(BG, \scr C)
\] the full subcategory of those objects on which $T$ acts invertibly.
If $\scr C$ is additive, denote by $\Fun(BG, \scr C)^r \subset \Fun(BG, \scr C)$ the full subcategory on which $|G|^2-T$ acts invertibly.
\end{definition}

\begin{proposition} \label{prop:BG-n-inverted}
Suppose that $n=|G|$ is invertible in $\scr C$, and assume that $\scr C$ is idempotent complete and stable.
\NB{Only use stable over additive to produce idempotents. Probably unnecessary.}.
\begin{enumerate}
\item Every object of $\Fun(BG, \scr C)$ splits as $A = A^0 \oplus A^r$, with $A^i \in \Fun(BG, \scr C)^i, i = 0, r$.
\item For $A \in \Fun(BG, \scr C)^0, B \in \Fun(BG, \scr C)^r$ we have $\Map(A, B) = 0 = \Map(B, A)$.
\item Let $p: BG \to *$ be the canonical map.
  Then $p^* \scr C \subset \Fun(BG, \scr C)^0$.
\item The functor $i^*: \Fun(BG, \scr C)^0 \to \scr C$ is an equivalence.
\item The functor $p^*: \scr C \to \Fun(BG, \scr C)^0$ is an equivalence.
\end{enumerate}
\end{proposition}
\begin{proof}
It follows from Lemma \ref{lemm:G-squaring} that $T/n^2$ defines an idempotent of every object, and thus $1-T/n^2$ defines a complementary idempotent \cite[Warning 1.2.4.8]{HA}.
Since these idempotents are by construction preserved by all morphisms, this (together with idempotent completeness of $\scr C$ and hence $\Fun(BG, \scr C)$ 
\cite[Corollaries 4.4.5.15 and 5.1.2.3]{HTT}) immediately implies (1) and (2).

(3) It follows from Remark \ref{rmk:T-reformulation} and \cite[Proposition 3.13, Example 3.12]{carmeli2018ambidexterity} applied to the cartesian square
\begin{equation*}
\begin{CD}
G @>>> * \\
@VVV   @ViVV \\
* @>i>> BG
\end{CD}
\end{equation*}
that $i^* T_A = n t_A$.
If $A = p^* A'$, then $t_A = n$ and so $i^* T_A = n^2$ is an isomorphism.
The result follows since $i^*$ is conservative.

(4) The functor $i^*: \Fun(BG, \scr C)^0 \to \scr C$ has a left (and right) adjoint given by $A \mapsto (i_*A)^0$.
Since $i^*$ is conservative, it suffices to show that for all $A \in \scr C$ we have $i^*(i_*(A)^0) \wequi A$.
Let $B \in \scr C$.
It suffices to show that we obtain an equivalence after applying $\pi_i \Map(B, \ph)$.
Since $\pi_i \Map(B, \ph): \scr C \to (\mathrm{Mod}_{\Z[1/n]})_{\le 0}$ is additive, this reduces to $\scr C = (\mathrm{Mod}_{\Z[1/n]})_{\le 0}$, i.e., the ordinary $1$-category of $\Z[1/n]$-modules.
This case is straightforward.

(5) Since $i^* p^* \wequi (pi)^* \wequi \id$, this follows from (3) and (4).
\end{proof}

\begin{corollary} \label{cor:compute-fixed}
Assumptions as in Proposition \ref{prop:BG-n-inverted}.
For every $A \in \Fun(BG, \scr C)$ we have canonical equivalences \[ A^{hG} \wequi i^* A^0 \wequi A_{hG}. \]
\end{corollary}
\begin{proof}
The functors $(\ph)^{hG}$ and $(\ph)_{hG}$ are given by $p_*$ and $p_\#$, respectively.
Proposition \ref{prop:BG-n-inverted} implies that $p^*$ can be identified with the inclusion $\Fun(BG, \scr C)^0 \to \Fun(BG, \scr C)$, and that its common right and left adjoint is given by $A \mapsto A^0$.
The result follows.
\end{proof}

\begin{corollary}\label{cor:universal-fixed}
Let $F: \scr C \to \scr D$ be an additive functor of additive, idempotent complete $\infty$-categories on which $|G|$ is invertible.
Then for any $A \in \Fun(BG, \scr C)$ we have \[ (FA)^{hG} \wequi F(A^{hG}) \wequi F(A_{hG}) \wequi (FA)_{hG}. \]
\end{corollary}
\begin{proof}
Any additive functor commutes with formation of $(\ph)^0$ (and $i^*$), so this follows from Corollary \ref{cor:compute-fixed}.
\end{proof}

\subsection{Galois descent}
Fix a finite group $G$.
In the category $\Fin_G$ we have the object $G$, with automorphism group canonically isomorphic to $G$ itself.
In this way we obtain an action of the group $G$ on $G \in \Fin_G$, i.e., $G: BG \to \Fin_G$.
Composing with the canonical functor $\Fin_G \to \Span(\Fin_G)$ we obtain an action of $G$ on $G \in \Span(\Fin_G)$.
In the category $\Span(\Fin_G)$ we also have the maps $a: * \to G$ and $b: G \to *$, corresponding to the spans \[ * \xleftarrow{} G \xrightarrow{\wequi} G \text{ and } G \xleftarrow{\wequi} G \to *. \]
It is straightforward to check that these refine to $G$-equivariant maps, where we let $G$ act trivially on $*$.
The composite $ba: * \to *$ is given by the span $* \leftarrow G \to *$ and denoted by $(G)$.

In the following result, we denote by $\SH(BG)$ the genuine $G$-equivariant stable category; see for example \cite[\S9]{norms}.
\begin{lemma} \label{lemm:tautological-fixed-points}
Consider the functor $\sigma: \Span(\Fin_G) \to \SH(BG)[1/|G|, 1/(G)]$.
Then $\sigma(a)$ exhibits $\sigma(*)$ as $\sigma(G)^{hG}$.
\end{lemma}
\begin{proof}
The composite $ab: G \to G$ is homotopic to $T_G/|G|$ by construction.
It follows from Corollary \ref{cor:compute-fixed} that in any category where $|G|$ is invertible, $G^{hG}$ is given by the summand of $G$ corresponding to the idempotent $ab/n$.

After inverting $|G|$, we obtain a splitting $G \wequi G^{hG} \oplus G'$ and we can write the maps $a, b$ in matrix form as $a = (a_1 a_2)$ and $b = (b_1 \, b_2)$.
Since $a$ factors through $G^{hG}$ we have $a_2 = 0$.
If we further invert $(G)$, then $(G) = ba = b_1a_1 + b_2 a_2 = b_1a_1$ is an equivalence, and also 
\[
\begin{pmatrix}\id & 0 \\ 0 & 0\end{pmatrix} = ab = \begin{pmatrix}a_1b_1 & a_1b_2 \\ a_2b_1 & a_2b_2\end{pmatrix},
\] so $a_1b_1 = \id$.
It follows that $a = a_1: \1 \to G^{hG}$ is an equivalence.
This was to be shown.
\end{proof}

Now let $S$ be a scheme and $S'/S$ a finitely presented finite étale scheme with a free and transitive $G$-action; i.e., a $G$-torsor, or in other words a $G$-Galois extension of $S$.
We then have a functor $\SH(BG) \to \SH(S)$.
In fact Grothendieck's Galois theory supplies us with a functor \[ g: \Fin_G \to \Sm_S, X \mapsto X \times_G S' \] and we have a natural transformation \cite[Proposition 10.6]{norms} \[ c: \SH(\ph) \to \SH(g(\ph)) \in \Fun(\Fin_G, \Cat). \]

\begin{lemma} \label{lemm:invert-G-plus}
The functors \[ \SH(BG) \to \SH(S) \to \SH(S)[1/2, 1/|G|]^+ \] and \[ \SH(BG) \to \SH(S) \to \DM(S, \Z[1/|G|]) \] invert the endomorphism $(G)$ of $\1 \in \SH(BG)$.
\end{lemma}
\begin{proof}
The problem is Zariski local on $S$, so we may assume that $S$ is affine.
Then since $S'$ is finitely presented, it is already defined (as a $G$-torsor) over some scheme $T$ under $S$ of finite type over $\Spec(\Z)$.\NB{i.e., basically a very easy case of Noetherian approximation}
We may thus assume that $S$ is finite dimensional.
In this situation pullback to fields is conservative \cite[Proposition B.3]{norms}, and so we reduce to the case when $S$ is the spectrum of a field.
In this case the morphism \[ A(G) \wequi [\1, \1]_{\SH(BG)} \to [\1, \1]_{\SH(S)[1/2]^+} \wequi \Z[1/2] \] is given by the rank homomorphism \cite[Theorem 10.12]{norms},
\NB{this is an overkill reference}, 
and similarly for \[ A(G) \wequi [\1, \1]_{\SH(BG)} \to [\1, \1]_{\DM(k)} \wequi \Z. \] 
The result follows.
\end{proof}

Now let $E \in \SH(S)$ and write $f: S' \to S$ for the canonical map.
The object $f_*f^* E$ acquires a $G$-action, coming from the action of $G$ on $S'$.

\begin{corollary} \label{cor:galois-descent}
\begin{enumerate}
\item The natural map $E \to f_* f^* E$ refines to a $G$-equivariant map (where the source is given the trivial $G$-action).
\item If $E \in \SH(S)[1/2,1/|G|]^+$ or $\DM(S, \Z[1/|G|])$ then the above map presents $E$ as $(f_*f^*E)^{hG}$.
  Moreover this limit diagram is preserved by any additive functor.
\end{enumerate}
\end{corollary}
\begin{proof}
We claim that the $G$-equivariant map $E \to f_* f^* E$ is given by $c(a) \wedge E$, where $a$ is the map from Lemma \ref{lemm:tautological-fixed-points}.
The same lemma, together with Lemma \ref{lemm:invert-G-plus} (and Corollary \ref{cor:universal-fixed}), then implies the result.

To prove the claim, we note that by the projection formula we indeed have $f_* f^* E \wequi (f_* f^* \1) \wedge E$, compatibly with the $G$-actions.
This reduces to $E = \1$.
The transformation $c: \SH(\ph) \to \SH(g(\ph))$ induces by passage to adjoints exchange transformations \[ f_\# c_{S'} \to c_S f_\# \text{ and } c_S f_* \to f_* c_{S'}. \]
The former is an equivalence essentially by construction, and the latter is an equivalence because $f_\# \wequi f_*$ on both sides (compatibly).
\NB{Say more here?}.
It follows that the $G$-action on $f_* f^* \1_{\SH(S)}$ is the same as the one induced by $c$ from the $G$-action on $f_* f^* \1_{\SH(BG)}$.
We have thus reduced to showing that the $G$-action on $f_* f^* \1 \wequi f_\# f^* \1 \wequi \Sigma^\infty_+ G \in \SH(BG)$ is the same as the one we constructed at the beginning of this subsection.
This can be verified directly by working in $\Span(\Fin_G)$.
\end{proof}

\begin{remark}
In the above two results, we could have more generally replaced $\DM$ by the category of modules over any ring spectrum in which the Hopf map $\eta$ is trivial, 
such as algebraic cobordism $\MGL$.
\end{remark}

\bibliographystyle{alpha}
\bibliography{bott}

\newcommand{\etalchar}[1]{$^{#1}$}
\begin{thebibliography}{EHK{\etalchar{+}}20}

\bibitem[Art83]{arthan1983localization}
R.~D. Arthan.
\newblock Localization of stable homotopy rings.
\newblock In {\em Mathematical Proceedings of the Cambridge Philosophical
  Society}, volume~93, pages 295--302. Cambridge University Press, 1983.

\bibitem[Ayo07]{ayoubthesis}
J.~Ayoub.
\newblock Les six op\'erations de {G}rothendieck et le formalisme des cycles
  \'evanescents dans le monde motivique. {I}, 2007.

\bibitem[Ayo14]{ayoub2014realisation}
J.~Ayoub.
\newblock La r{\'e}alisation {\'e}tale et les op{\'e}rations de {G}rothendieck.
\newblock {\em Ann. Sci. {\'E}c. Norm. Sup{\'e}r.(4)}, 47(1):1--145, 2014.

\bibitem[Bac17]{bachmann-very-effective}
T.~Bachmann.
\newblock The generalized slices of hermitian {$K$}-theory.
\newblock {\em Journal of Topology}, 10(4):1124--1144, 2017.
\newblock \href{https://arxiv.org/abs/1610.01346}{arXiv:1610.01346}.

\bibitem[Bac18a]{bachmann-real-etale}
T.~Bachmann.
\newblock Motivic and real étale stable homotopy theory.
\newblock {\em Compositio Mathematica}, 154(5):883–917, 2018.
\newblock \href{https://arxiv.org/abs/1608.08855}{arXiv:1608.08855}.

\bibitem[Bac18b]{bachmann-hurewicz}
T.~Bachmann.
\newblock On the conservativity of the functor assigning to a motivic spectrum
  its motive.
\newblock {\em Duke Math. J.}, 167(8):1525--1571, 06 2018.
\newblock \href{https://arxiv.org/abs/1506.07375}{arXiv:1506.07375}.

\bibitem[Bac21]{bachmann-SHet}
T.~Bachmann.
\newblock Rigidity in étale motivic stable homotopy theory.
\newblock {\em Algebraic \& Geometric Topology}, 21(1):173--209, 2021.
\newblock \href{https://arxiv.org/abs/1810.08028}{arXiv:1810.08028}.

\bibitem[BBX]{bbx}
T.~Bachmann, R.~Burklund, and Z.~Xu.
\newblock Topological reconstruction of motivic categories.
\newblock In Preparation.

\bibitem[BCR13]{bochnak2013real}
J.~Bochnak, M.~Coste, and M.~F. Roy.
\newblock {\em Real algebraic geometry}, volume~36.
\newblock Springer Science \& Business Media, 2013.

\bibitem[BH20a]{bachmann-hopkins}
T.~Bachmann and M.~J. Hopkins.
\newblock $\eta$-periodic motivic stable homotopy theory over fields.
\newblock 2020.

\bibitem[BH20b]{norms}
T.~Bachmann and M.~Hoyois.
\newblock Norms in motivic homotopy theory.
\newblock {\em to appear in Asterisque}, 2020.
\newblock \href{https://arxiv.org/abs/1711.03061}{arXiv:1711.03061}.

\bibitem[BH21]{bachmann-SHet2}
T.~Bachmann and M.~Hoyois.
\newblock Remarks on étale motivic stable homotopy theory.
\newblock \href{https://arxiv.org/abs/2104.06002}{arXiv:2104.06002}, 2021.

\bibitem[BKS{\O}15]{bkso}
A.~J. Berrick, M.~Karoubi, M.~Schlichting, and P.~A. {\O}stv{\ae}r.
\newblock The {H}omotopy {F}ixed {P}oint {T}heorem and the
  {Q}uillen-{L}ichtenbaum conjecture in {H}ermitian {$K$}-theory.
\newblock {\em Adv. Math.}, 278:34--55, 2015.

\bibitem[Blo94]{bloch-moving}
S.~Bloch.
\newblock The moving lemma for higher {C}how groups.
\newblock {\em J. Algebraic Geom.}, 3(3):537--568, 1994.

\bibitem[BS20]{behrens2019c_2}
M.~Behrens and J.~Shah.
\newblock {$C_2$}-equivariant stable homotopy from real motivic stable
  homotopy.
\newblock {\em Ann. K-Theory}, 5(3):411--464, 2020.

\bibitem[CD16]{etalemotives}
D.-C. Cisinski and F.~D\'eglise.
\newblock \'{E}tale motives.
\newblock {\em Compos. Math.}, 152(3):556--666, 2016.

\bibitem[CD19]{triangulated-mixed-motives}
D.-C. Cisinski and F.~D\'{e}glise.
\newblock {\em Triangulated categories of mixed motives}.
\newblock Springer Monographs in Mathematics. Springer, Cham, [2019] \copyright
  2019.

\bibitem[CM21]{clausen-mathew}
D.~Clausen and A.~Mathew.
\newblock Hyperdescent and \'{e}tale {$K$}-theory.
\newblock {\em Invent. Math.}, 225(3):981--1076, 2021.

\bibitem[CSY18]{carmeli2018ambidexterity}
S.~Carmeli, T.~M. Schlank, and L.~Yanovski.
\newblock Ambidexterity in $t(n)$-local stable homotopy theory.
\newblock arXiv:1811.02057, 2018.

\bibitem[DI05]{cell}
D.~Dugger and D.~C. Isaksen.
\newblock Motivic cell structures.
\newblock {\em Algebr. Geom. Topol.}, 5:615--652, 2005.

\bibitem[Dug14]{dugger2014coherence}
D.~Dugger.
\newblock Coherence for invertible objects and multigraded homotopy rings.
\newblock {\em Algebraic \& Geometric Topology}, 14(2):1055--1106, 2014.

\bibitem[EHK{\etalchar{+}}19]{EHKSY}
E.~Elmanto, M.~Hoyois, A.~A. Khan, V.~Sosnilo, and M.~Yakerson.
\newblock Motivic infinite loop spaces.
\newblock to appear in Cambridge J. Math., arXiv:1711.05248, 2019.

\bibitem[EHK{\etalchar{+}}20]{EHKSY2}
E.~Elmanto, M.~Hoyois, A.~A. Khan, V.~Sosnilo, and M.~Yakerson.
\newblock Framed transfers and motivic fundamental classes.
\newblock {\em J. Topol.}, 13(2):460--500, 2020.

\bibitem[EK20]{elmanto2018perfection}
E.~Elmanto and A.~A. Khan.
\newblock Perfection in motivic homotopy theory.
\newblock {\em Proc. Lond. Math. Soc. (3)}, 120(1):28--38, 2020.

\bibitem[EKM08]{MR2427530}
R.~Elman, N.~Karpenko, and A.~Merkurjev.
\newblock {\em The algebraic and geometric theory of quadratic forms},
  volume~56 of {\em American Mathematical Society Colloquium Publications}.
\newblock American Mathematical Society, Providence, RI, 2008.

\bibitem[EL99]{elman-lum-2cohom}
R.~Elman and C.~Lum.
\newblock On the cohomological 2-dimension of fields.
\newblock {\em Communications in Algebra}, 27(2):615--620, 1999.

\bibitem[ELS{\O}17]{elso}
E.~Elmanto, M.~Levine, M.~Spitzweck, and P.~A. {\O}stv{\ae}r.
\newblock Algebraic cobordism and \'etale cohomology.
\newblock arXiv:1711.06258, 2017.
\newblock to appear in Geom. Topol.

\bibitem[FS02]{friedlander-suslin}
Eric~M. Friedlander and Andrei Suslin.
\newblock The spectral sequence relating algebraic {$K$}-theory to motivic
  cohomology.
\newblock {\em Ann. Sci. \'{E}cole Norm. Sup. (4)}, 35(6):773--875, 2002.

\bibitem[Gab92]{gabber-rigidity}
O.~Gabber.
\newblock {$K$}-theory of {H}enselian local rings and {H}enselian pairs.
\newblock 126:59--70, 1992.

\bibitem[GJ09]{goerss2009simplicial}
P.~G. Goerss and J.~F. Jardine.
\newblock {\em Simplicial homotopy theory}.
\newblock Springer Science \& Business Media, 2009.

\bibitem[GP21]{garkusha2014framed}
G.~Garkusha and I.~Panin.
\newblock Framed motives of algebraic varieties (after {V}. {V}oevodsky).
\newblock {\em J. Amer. Math. Soc.}, 34(1):261--313, 2021.

\bibitem[GSZ16]{gsz}
S.~Gille, S.~Scully, and C.~Zhong.
\newblock Milnor-{W}itt {$K$}-groups of local rings.
\newblock {\em Adv. Math.}, 286:729--753, 2016.

\bibitem[Har17]{harpaz2017ambidexterity}
Y.~Harpaz.
\newblock Ambidexterity and the universality of finite spans.
\newblock arXiv:1703.09764, 2017.

\bibitem[Hoy17]{hoyois-sixops}
M.~Hoyois.
\newblock The six operations in equivariant motivic homotopy theory.
\newblock {\em Adv. Math.}, 305:197--279, 2017.

\bibitem[Hoy21]{framed-loc}
M.~Hoyois.
\newblock The localization theorem for framed motivic spaces.
\newblock {\em Compos. Math.}, 157(1):1--11, 2021.

\bibitem[Jar10]{jardine}
J.~F. Jardine.
\newblock {\em Generalized etale cohomology theories}.
\newblock Modern Birkh\"{a}user Classics. Birkh\"{a}user/Springer Basel AG,
  Basel, 2010.
\newblock Reprint of the 1997 edition [MR1437604].

\bibitem[Lev00]{levine2000inverting}
M.~Levine.
\newblock Inverting the motivic bott element.
\newblock {\em K-theory}, 19(1):1--28, 2000.

\bibitem[Lev01]{levine-techniques}
M.~Levine.
\newblock Techniques of localization in the theory of algebraic cycles.
\newblock {\em J. Algebraic Geom.}, 10(2):299--363, 2001.

\bibitem[Lev11]{levine-gw}
M.~Levine.
\newblock The slice filtration and {G}rothendieck-{W}itt groups.
\newblock {\em Pure Appl. Math. Q.}, 7(4, Special Issue: In memory of Eckart
  Viehweg):1543--1584, 2011.

\bibitem[Lev13]{levine2013convergence}
M.~Levine.
\newblock Convergence of {V}oevodsky's slice tower.
\newblock {\em Documenta Mathematica}, 18:907--941, 2013.

\bibitem[Lev18]{levine-appreciation}
M.~Levine.
\newblock Vladimir {V}oevodsky---an appreciation.
\newblock {\em Bull. Amer. Math. Soc. (N.S.)}, 55(4):405--425, 2018.

\bibitem[Lur09]{HTT}
J.~Lurie.
\newblock {\em Higher topos theory}.
\newblock Number 170. Princeton University Press, 2009.

\bibitem[Lur16]{HA}
J.~Lurie.
\newblock Higher algebra, May 2016.

\bibitem[Lur18]{SAG}
J.~Lurie.
\newblock Spectral algebraic geometry.
\newblock February 2018.

\bibitem[LYZ19]{levine2013algebraic}
M.~Levine, Y.~Yang, and G.~Zhao.
\newblock Algebraic elliptic cohomology theory and flops {I}.
\newblock {\em Math. Ann.}, 375(3-4):1823--1855, 2019.

\bibitem[Mil80]{MilneEtaleCohomology}
J.~S. Milne.
\newblock {\em Etale {C}ohomology ({PMS}-33)}.
\newblock Princeton mathematical series. Princeton University Press, 1980.

\bibitem[MNN17]{mnn}
A.~Mathew, N.~Naumann, and J.~Noel.
\newblock Nilpotence and descent in equivariant stable homotopy theory.
\newblock {\em Adv. Math.}, 305:994--1084, 2017.

\bibitem[Mor]{morel-book}
F.~Morel.
\newblock {\em ${\bf A}^{1}$-algebraic topology over a field}.
\newblock Lecture Notes in Mathematics, Vol. 2052. Springer Verlag.

\bibitem[Mor04a]{morel-trieste}
F.~Morel.
\newblock An introduction to {$\Bbb A^{1}$}-homotopy theory.
\newblock pages 357--441, 2004.

\bibitem[Mor04b]{morel-pi0}
F.~Morel.
\newblock On the motivic {$\pi_0$} of the sphere spectrum.
\newblock In {\em Axiomatic, enriched and motivic homotopy theory}, volume 131
  of {\em NATO Sci. Ser. II Math. Phys. Chem.}, pages 219--260. Kluwer Acad.
  Publ., Dordrecht, 2004.

\bibitem[Mor04c]{morel-ideal}
F.~Morel.
\newblock Sur les puissances de l'id\'{e}al fondamental de l'anneau de {W}itt.
\newblock {\em Comment. Math. Helv.}, 79(4):689--703, 2004.

\bibitem[Mor05]{morel-conn}
F.~Morel.
\newblock The stable {${\mathbb A}^1$}-connectivity theorems.
\newblock {\em K-Theory}, 35(1-2):1--68, 2005.

\bibitem[Mor12]{A1-alg-top}
F.~Morel.
\newblock {\em $\mathbb{A}^1$-Algebraic Topology over a Field}.
\newblock Lecture Notes in Mathematics. Springer Berlin Heidelberg, 2012.

\bibitem[MVW06]{mvw}
C.~Mazza, V.~Voevodsky, and C.~A. Weibel.
\newblock {\em Lecture notes on motivic cohomology}, volume~2 of {\em Clay
  Mathematics Monographs}.
\newblock Amer. Math. Society, Providence, RI; Clay Mathematics Institute,
  Cambridge, MA, 2006.

\bibitem[Neu13]{neukirch2013algebraic}
J.~Neukirch.
\newblock {\em Algebraic number theory}, volume 322.
\newblock Springer Science \& Business Media, 2013.

\bibitem[Oka84]{oka1984multiplications}
S.~Oka.
\newblock Multiplications on the {M}oore spectrum.
\newblock {\em Memoirs of the Faculty of Science, Kyushu University. Series A,
  Mathematics}, 38(2):257--276, 1984.

\bibitem[OR20]{ormsby-rondigs}
K.~Ormsby and O.~R\"{o}ndigs.
\newblock The homotopy groups of the {$\eta$}-periodic motivic sphere spectrum.
\newblock {\em Pacific J. Math.}, 306(2):679--697, 2020.

\bibitem[{\O}st03]{pa-thesis}
P.~A. {\O}stv{\ae}r.
\newblock \'{E}tale descent for real number fields.
\newblock {\em Topology}, 42(1):197--225, 2003.

\bibitem[OVV07]{ovv}
D.~Orlov, A.~Vishik, and V.~Voevodsky.
\newblock An exact sequence for {$K^M_\ast/2$} with applications to quadratic
  forms.
\newblock {\em Ann. of Math. (2)}, 165(1):1--13, 2007.

\bibitem[R{\O}05]{ro}
A.~Rosenschon and P.~A. {\O}stv{\ae}r.
\newblock The homotopy limit problem for two-primary algebraic {$K$}-theory.
\newblock {\em Topology}, 44(6):1159--1179, 2005.

\bibitem[R{\O}08a]{rigidity-in-motivic-homotopy-theory}
O.~R{\"o}ndigs and P.~A. {\O}stv{\ae}r.
\newblock Rigidity in motivic homotopy theory.
\newblock {\em Mathematische Annalen}, 341(3):651--675, 2008.

\bibitem[R{\O}08b]{RondigsModules}
O.~Röndigs and P.~A. {\O}stv{\ae}r.
\newblock Modules over motivic cohomology.
\newblock {\em Advances in Mathematics}, 219(2):689 -- 727, 2008.

\bibitem[R{\"o}n19]{rondigs2019remarks}
O.~R{\"o}ndigs.
\newblock Remarks on motivic {M}oore spectra.
\newblock {\em arXiv preprint arXiv:1910.00834}, 2019.

\bibitem[RS{\O}18]{rso-solves}
O.~R\"{o}ndigs, M.~Spitzweck, and P.~A. {\O}stv{\ae}r.
\newblock The motivic {H}opf map solves the homotopy limit problem for
  {$K$}-theory.
\newblock {\em Doc. Math.}, 23:1405--1424, 2018.

\bibitem[RS{\O}19]{1-line}
O.~R\"{o}ndigs, M.~Spitzweck, and P.~A. {\O}stv{\ae}r.
\newblock The first stable homotopy groups of motivic spheres.
\newblock {\em Ann. of Math. (2)}, 189(1):1--74, 2019.

\bibitem[Ser02]{serre-gal}
J.-P. Serre.
\newblock {\em Galois cohomology}.
\newblock Springer Monographs in Mathematics. Springer-Verlag, Berlin, english
  edition, 2002.
\newblock Translated from the French by Patrick Ion and revised by the author.

\bibitem[Spi18]{spitzweck2012commutative}
M.~Spitzweck.
\newblock A commutative {$\Bbb P^1$}-spectrum representing motivic cohomology
  over {D}edekind domains.
\newblock {\em M\'{e}m. Soc. Math. Fr. (N.S.)}, (157):110, 2018.

\bibitem[{Sta}17]{stacks}
The {Stacks Project Authors}.
\newblock The stacks project.
\newblock 2017.

\bibitem[Sus83]{suslin-closed}
A.~Suslin.
\newblock On the {$K$}-theory of algebraically closed fields.
\newblock {\em Invent. Math.}, 73(2):241--245, 1983.

\bibitem[SV00]{MR1744945}
A.~Suslin and V.~Voevodsky.
\newblock Bloch-{K}ato conjecture and motivic cohomology with finite
  coefficients.
\newblock In {\em The arithmetic and geometry of algebraic cycles ({B}anff,
  {AB}, 1998)}, volume 548 of {\em NATO Sci. Ser. C Math. Phys. Sci.}, pages
  117--189. Kluwer Acad. Publ., Dordrecht, 2000.

\bibitem[Tho85]{aktec}
R.~W. Thomason.
\newblock Algebraic {$K$}-theory and \'etale cohomology.
\newblock {\em Ann. Sci. \'Ecole Norm. Sup. (4)}, 18(3):437--552, 1985.

\bibitem[Tho91]{thomason-icm}
R.~W. Thomason.
\newblock The local to global principle in algebraic {$K$}-theory.
\newblock pages 381--394, 1991.

\bibitem[TT90]{tt}
R.~W. Thomason and T.~Trobaugh.
\newblock Higher algebraic {$K$}-theory of schemes and of derived categories.
\newblock 88:247--435, 1990.

\bibitem[Voe02]{voe-open}
V.~Voevodsky.
\newblock Open problems in the motivic stable homotopy theory. {I}.
\newblock 3:3--34, 2002.

\bibitem[Voe11]{voevodsky-BK}
V.~Voevodsky.
\newblock On motivic cohomology with $\mathbf{Z}/l$-coefficients.
\newblock {\em Annals of mathematics}, 174(1):401--438, 2011.

\end{thebibliography}

\end{document}